\documentclass[11pt,a4paper]{article}
\usepackage{fullpage}
\usepackage[utf8]{inputenc}
\usepackage[T1]{fontenc}
\ifdefined\XeTeXversion\PassOptionsToPackage{xetex}{hyperref}\else\ifdefined\pdfoutput\pdfoutput=1\fi\fi

\usepackage{amsmath}
\usepackage{amsfonts}
\usepackage{amssymb}
\usepackage{mathtools}

\usepackage{lmodern}
\usepackage{mathrsfs}
\ifdefined\XeTeXversion
  \IfFileExists{dsfont.sty}{\usepackage{dsfont}\let\mathbbm\mathds}{}
\fi
\ifdefined\mathbbm\else
  \IfFileExists{bbm.sty}{\usepackage{bbm}}{%
    \IfFileExists{dsfont.sty}{\usepackage{dsfont}\let\mathbbm\mathds}{\let\mathbbm\mathbf}}
\fi

\usepackage[parfill]{parskip}
\usepackage{array}

\usepackage{graphicx}
\usepackage[font=small,labelfont=bf]{caption}
\usepackage{tikz-cd}

\usepackage{verbatim}
\usepackage{amsthm}

\usepackage[colorlinks=true, allcolors=blue]{hyperref}
\usepackage[nameinlink]{cleveref}

\newtheorem{theorem}{Theorem}

\newtheorem{corollary}{Corollary}[section]
\newtheorem{lemma}{Lemma}

\theoremstyle{remark}
\newtheorem{rem}{Remark}[section]

\theoremstyle{definition}

\newcommand{\ZZ}{\mathbb{Z}}
\newcommand{\NN}{\mathbb{N}}
\newcommand{\irr}{\mathfrak{Irr}}
\newcommand{\comp}{\mathrm{Comp}}
\newcommand{\ra}{\rightarrow}
\newcommand{\orb}{\mathrm{Orb}}
\newcommand{\oo}{\mathbf{O}}
\newcommand{\ext}{\mathrm{ext}}
\newcommand{\id}{\mathrm{Ind}}
\newcommand{\FF}{\mathrm{F}}
\newcommand{\CC}{\mathbb{C}}
\newcommand{\ed}{\mathrm{End}}
\newcommand{\cusp}{\mathrm{cusp}}
\renewcommand{\hom}{\mathrm{Hom}}
\renewcommand{\ext}{\mathrm{Ext}}

\newcommand{\ql}{{\overline{\mathbb{Q}}_\ell}}

\newcommand{\rep}{\mathfrak{Rep}}
\newcommand{\De}{\Delta}

\newcommand{\gl}{\mathbf{GL}}
\newcommand{\op}{\overline{\bP}}
\newcommand{\supp}{\mathrm{supp}}
\newcommand{\db}{\mathbf{d}}

\newcommand{\hra}{\hookrightarrow}
\newcommand{\LL}{\mathrm{Z}}

\newcommand{\pard}{\mathfrak{d}}
\newcommand{\Md}{\mathrm{Mod}}

\newcommand{\cn}{\mathbf{N}_n}
\newcommand{\btimes}{\overline{\times}}
\newcommand{\hyp}{\mathrm{Hyp}}

\newcommand{\bJ}{\Bar{J}}
\newcommand{\bLambda}{\Bar{\Lambda}}
\newcommand{\ain}[3]{#1\in \{#2,\ldots,#3\}}
\newcommand{\bX}{\mathbf{X}}
\newcommand{\gm}{{\mathbf{G}_m}}
\newcommand{\mhm}{\mathrm{MHM}}
\newcommand{\DD}{\mathbb{D}}
\newcommand{\QQ}{\mathbb{Q}}
\newcommand{\tep}{\tilde{\mathbf{E}}_\db}

\newcommand{\cus}{\mathfrak{Cus}}
\newcommand{\rat}{\mathrm{Rat}}

\newcommand{\gr}{\mathrm{gr}}
\newcommand{\ep}{\mathbf{E}_\db}
\newcommand{\epo}{\mathbf{E}_{\db^1}}
\newcommand{\epk}{\mathbf{E}_{\db^k}}
\newcommand{\grdim}{\mathrm{grdim}}
\newcommand{\mul}{\mathcal{M}ul}
\newcommand{\fm}{\mathfrak{m}}
\newcommand{\bV}{\overline{V}}
\newcommand{\Z}{\mathrm{Z}}
\newcommand{\I}{\mathrm{I}}
\newcommand{\fn}{\mathfrak{n}}
\newcommand{\Ir}{\mathrm{I^{gr}}}
\newcommand{\cf}{\mathcal{F}}
\newcommand{\sem}{\mathrm{Semis}}

\newcommand{\ic}[1]{\mathrm{IC}(#1,\underline{\CC}_{#1})}

\newcommand{\cS}{\mathcal{S}}

\newcommand{\conv}{\mathrm{Conv}^+}
\newcommand{\grtim}{\times_{\mathrm{gr}}}
\newcommand{\bq}{\mathbf{q}}

\newcommand{\real}{\mathrm{real}}

\newcommand{\twh}{[t,t^{-1}]}
\newcommand{\ct}{\mathcal{T}}
\newcommand{\cb}{\mathbf{Fl}_n}

\newcommand{\grtimt}{\tilde{\times}_\gr}
\newcommand{\coh}{\mathrm{Coh}}
\newcommand{\bY}{\mathbf{Y}}

\newcommand{\uc}{\underline{\CC}}
\newcommand{\HH}{\mathbb{H}}
\newcommand{\bG}{\mathbf{G}}
\newcommand{\cuss}{\cus/_\sim}
\newcommand{\bH}{\mathbf{H}}
\newcommand{\perv}{\mathrm{Perv}}
\newcommand{\tH}{\mathrm{H}}
\newcommand{\bU}{\mathbf{U}}

\newcommand{\bB}{\mathbf{B}}
\newcommand{\bM}{\mathbf{M}}
\newcommand{\bP}{\mathbf{P}}
\newcommand{\siso}{\mathrm{SS}}
\newcommand{\gdb}{\mathbf{G}_\db}
\title{Generalized Jantzen filtration, hyperbolic restriction, and characteristic cycles}
\author{Johannes Droschl\\ \href{mailto:au800840@uni.au.dk}{\nolinkurl{au800840@uni.au.dk}}}
\date{\today}

\begin{document}
\maketitle
\begin{abstract}
    In this paper we give a geometric description of the behavior of analytic intertwining operators between parabolically induced representations of $\gl_n(\FF)$, where $\FF$ is a non-archimedean local field, in terms of the hyperbolic localization functors of Braden. As a consequence, we can show that the image of the normalized intertwining operator for two irreducible inducing representations is semi-simple and give two-sided bounds on the pole order of the original operator in terms of generic Hom and Ext dimensions for the preprojective algebra. When the upper bound is attained, the singular support of each constituent of the image contains the generic-extension component appearing in the conjecture of Lapid and Mínguez. The theory of mixed Hodge modules plays a crucial role in the proofs.
\end{abstract}
\tableofcontents
\section{Introduction}
Let $\FF$ be a non-archimedean local field with absolute value $\lvert-\rvert$ and $\bG$ a reductive group over $\FF$. The study of irreducible smooth representations of $\bG(\FF)$ lies at the heart of the local Langlands program. In the foundational works of Bernstein and Zelevinsky \cite{BerZel76}, \cite{BerZel77}, \cite{Zel80}, these representations were classified in the case $\bG=\gl_n$. At the same time, the theory of types of Bushnell and Kutzko \cite{BushnellKutzko1998}, \cite{BushnellKutzko1999} allows one to capture the complete behavior of the category of smooth finite-length representations of $\gl_n(\FF)$ by categories of modules over affine Hecke algebras. On the other hand, the geometric representation theory employed in \cite{KazhdanLusztig1987}, \cite{ChrissGinzburg1997} managed to answer many questions in the study of Hecke algebras satisfactorily in terms of the behavior of intersection-cohomology complexes on the nilpotent cone of the dual group. Combining these results, one can find solutions to many problems arising in the representation theory of $\gl_n(\FF)$. Nevertheless, some open conjectures have resisted the efforts of the field for many years. In this paper we will focus on one such problem, namely the computation of analytic intertwining operators between parabolically induced representations. Many of the constructions in this paper will also work, in slightly altered form, for more general groups. We choose for now to focus on the case $\gl_n$, since it allows us to tell a more complete story. In particular, our techniques allow us to make progress on the conjectures of Lapid and Mínguez \cite{LapMin25} and the author \cite{Dro25}. Note that the behavior of such intertwining operators has been studied in a wide variety of settings, in particular for quantum groups. See for example \cite{GeissLeclercSchroer2006}, \cite{HernandezLeclerc2010}, and \cite{KKKO}.

The paper is divided into three sections, one in which we construct the generalized Jantzen filtration, one on hyperbolic restriction and the generalized Rogawski conjecture, and a final one in which we combine the above constructions with Saito's theory of mixed Hodge modules. Before we discuss their contents in more detail, we recall some necessary notation. Moreover, we will focus the exposition in the introduction on a particular subcase of the problem, which nevertheless contains the germ of the general case. We denote by $\irr_n$ the set of isomorphism classes of irreducible representations of $\gl_n(\FF)$ and set $\irr\coloneq\bigcup_{n\in\NN}\irr_n$. Let $\alpha=(\alpha_1,\ldots,\alpha_k)$ be a partition of $n$. For $\pi_i\in\irr_{\alpha_i},\, \ain{i}{1}{k}$, we denote as usual the normalized parabolically induced representation of $\gl_n(\FF)$ by \[\pi_1\times\ldots\times\pi_k.\] 
We also write, for a smooth character $\chi\colon\FF^\times\ra\CC$ and $\pi\in\irr_n$, $\pi\chi\coloneq\pi\otimes\chi(\det_n)$.
For $\lambda\in\ZZ^k$ generic, we recall the intertwining operator (\textit{cf.} \cite[§IV]{Waldspurger2003Plancherel})
\[J_{\pi_1,\ldots,\pi_k}(s)\colon \pi_k\lvert-\rvert^{s\lambda_k}\times\ldots\times\pi_1\lvert-\rvert^{s\lambda_1}\ra\pi_1\lvert-\rvert^{s\lambda_1}\times\ldots\times\pi_k\lvert-\rvert^{s\lambda_k}.\]
We also consider for $\db\in\bigoplus_{i\in\ZZ}\NN$, which we from now on will refer to as a dimension vector, the subset $\irr_{\db}$ of irreducible representations which appear as subquotients of
\[\ldots\times \overbrace{\lvert-\rvert^i\times\ldots\times\lvert-\rvert^i}^{\db_i}\times\ldots.\]
For $\db^1,\db^2$ two dimension vectors we set
\[\langle\db^1,\db^2\rangle\coloneq\sum_{i\in\ZZ}\db_i^1\db_i^2-\db_i^1\db_{i+1}^2.\]
Given $\pi_1\in\irr_{\db^1},\pi_2\in\irr_{\db^2}$ and $\lambda\in\ZZ^2$ generic, we study in \Cref{S:JanGeo} the following increasing filtration $\{V_i(\pi_2,\pi_1)\}_{i\in\ZZ}$ of $\pi_2\times\pi_1$
\[V_i(\pi_2,\pi_1)\coloneq \{f\in\pi_2\times\pi_1: \lim_{s\ra 0}s^{i+\langle\db^1,\db^2\rangle}J_{\pi_1,\pi_2}(s)(f_s)\text{ exists}\},\]
where the limit is required to exist for some holomorphic section germ $f_s\in \pi_2\lvert-\rvert^{s\lambda_2}\times\pi_1\lvert-\rvert^{s\lambda_1}$ with $f_0=f$, in the compact realization of \cite[§IV]{Waldspurger2003Plancherel}.
In the case of standard modules this was, up to a shift in the grading, already considered in \cite{Rogawski1985}.
We denote \[\pi_1\times\pi_2[i]\coloneq V_{i-1}(\pi_1,\pi_2)\backslash V_i(\pi_1,\pi_2),\] which is non-zero only for finitely many $i$.
Note that the graded piece for the maximal $i$ such that $\pi_1\times\pi_2[i]$ does not vanish is the image of the normalized intertwining operator \[J_{\pi_2,\pi_1}\coloneq s^{i+\langle\db^2,\db^1\rangle}J_{\pi_2,\pi_1}(s)\rvert_{s=0}\colon\pi_1\times\pi_2\ra\pi_2\times\pi_1,\] and we write $\Lambda(\pi_2,\pi_1)\coloneq i+\langle\db^2,\db^1\rangle$. 
In \Cref{S:Jantzen} we define $\pi_1\times\pi_2[i]$ for all $\pi_1,\pi_2\in\irr$ and obtain the following.
In the theorem we denote by $[-]_{ss}$ the semi-simplification of a representation.
\begin{theorem}
    The $\ZZ[t,t^{-1}]$-module $\ZZ[\irr][t,t^{-1}]$ admits an associative product
    \[\grtimt\colon \ZZ[\irr][t,t^{-1}]\times \ZZ[\irr][t,t^{-1}]\ra \ZZ[\irr][t,t^{-1}],\,(\pi_1,\pi_2)\mapsto \sum_{i\in\ZZ}[\pi_1\times\pi_2[i]]_{ss}t^i.\]
    Moreover, $\pi_1\grtimt\pi_2=\overline{\pi_2\grtimt\pi_1}$, where 
    \[\overline{(-)}\colon \ZZ[\irr][t,t^{-1}]\ra \ZZ[\irr][t,t^{-1}],\, t\mapsto t^{-1}.\]
\end{theorem}
In particular, setting $t=1$, one recovers the usual product \[\times\colon\ZZ[\irr]\times\ZZ[\irr]\ra\ZZ[\irr]\] given by parabolic induction. This product is commutative on the Grothendieck group. The associativity is fundamentally a geometric property of the underlying spaces, in our case moduli spaces of quiver representations. To the knowledge of the author, there does not exist a purely representation-theoretical proof of associativity.

In the second part, \Cref{S:JanGeo}, we reinterpret the above construction geometrically. For this we consider, for $\db$ a dimension vector, the space
\[\ep(\CC)\coloneq \{(T_i)_{i\in\ZZ}:T_i\colon \CC^{\db_i}\ra \CC^{\db_{i+1}},\,\CC\text{-linear}\},\]
on which 
\[\gdb(\CC)\coloneq \bigtimes_{i\in\ZZ} \gl_{\db_i}(\CC)\] acts by conjugation.
Denote by $\orb_\db$ the set of $\gdb$-orbits.
Zelevinsky's classification \cite{Zel80} gives a bijection
\[\Z\colon \orb_\db\ra\irr_\db,\]
sending the unique open orbit to the unique unramified representation in $\irr_\db$.
Consider now two dimension vectors $\db^1,\db^2$ with $\db=\db^1+\db^2$.
In \Cref{S:JanGeo} we construct a cocharacter $\chi\colon\gm\ra \gdb$ such that its centralizer is $\bG_{\db^1}\times \bG_{\db^2}$ and the fixed-point locus in $\ep$ is isomorphic to $\mathbf{E}_{\db^1}\times \mathbf{E}_{\db^2}$.
We denote for $\oo\in\orb_\db$ by $\ic{\oo}$ the simple perverse sheaf corresponding to the trivial local system on $\oo(\CC)$ and by $\sem(\ep,\gdb)$ the $\ZZ[t,t^{-1}]$-module spanned by all shifts of $\{\ic{\oo}:\oo\in\orb_\db\}$, where $t$ acts on a sheaf $\cf$ via $t\cdot\cf=\cf[-1]$.
Thanks to \cite{KazhdanLusztig1987} there exists a natural isomorphism of $\ZZ[t,t^{-1}]$-modules
\[\ct_\db\colon \sem(\ep,\gdb)\ra \ZZ[\irr_\db][t,t^{-1}]^*,\]
where $(-)^*$ denotes the dual module and $\ic{\oo}$ is sent to the indicator function of $\Z(\oo)$.
We will denote by $\ct_{\db^1,\db^2}$ the analogous bijection involving $\irr_{\db^1}\times\irr_{\db^2}$ and $\mathbf{E}_{\db^1}\times \mathbf{E}_{\db^2}$.
The hyperbolic restriction functor of \cite{braden2003hyperbolic} constructed from the action of $\gm$ now plays a crucial role.
\[\hyp^+\colon \sem(\ep,\gdb)\ra \sem(\mathbf{E}_{\db^1}\times \mathbf{E}_{\db^2},\bG_{\db^1}\times \bG_{\db^2}).\]
The main theorem of \Cref{S:JanGeo} is then the following.
\begin{theorem}
    The following diagram commutes.
    \[\begin{tikzcd}
        \sem(\ep,\gdb)\arrow[r,"\hyp^+"]\arrow[d,"\ct_\db"]& \sem(\mathbf{E}_{\db^1}\times \mathbf{E}_{\db^2},\bG_{\db^1}\times \bG_{\db^2})\arrow[d,"\ct_{\db^1,\db^2}"]\\
        \ZZ[\irr_\db][t,t^{-1}]^* \arrow[r,"\grtimt^*"]&\ZZ[\irr_{\db^1}\times\irr_{\db^2}][t,t^{-1}]^*
    \end{tikzcd}\]
    Moreover, $\pi_1\times\pi_2[i]$ is semi-simple for all $i\in\ZZ$. In particular, the image of the intertwining operator
    \[J_{\pi_1,\pi_2}\colon \pi_2\times\pi_1\ra\pi_1\times\pi_2\] is semi-simple.
\end{theorem}
By the change of parameter in the proof of \Cref{L:ind}, the normalized intertwining operator
\[J_{\pi_1,\pi_2}\colon\pi_2\times\pi_1\ra\pi_1\times\pi_2\]
is, up to a scalar, independent of the chosen generic $\lambda\in\ZZ^2$.
As always, a semi-simplicity statement hints at an appearance of the Decomposition theorem of \cite{BBD1982}, and this is also here the case.

Appendix~\ref{app:intertwiner} realizes the type $A$ intertwining operator as a comparison between the two hyperbolic restrictions after equivariant localization. Its elementary divisors on the original lattices give the degrees in the theorem. The corresponding statement for standard modules is the Rogawski conjecture of \cite{Rogawski1985}, for which a proof was announced by Ginzburg in \cite{Ginzburg1987} without details and shown for degenerate affine Hecke algebras in type $A$ in \cite{Suzuki1998}. Recently, it has been proven by \cite{Ciubotaru2026Jantzen}, see also \cite{FujitaHernandez2026}, using the hard Lefschetz theorem.
Recall that for $\pi=\Z(\oo)$ one can construct a standard module $\I(\pi)=\pi_1\times\ldots\times\pi_k$ with unique irreducible quotient $\pi$. Here $\pi_a$ is the character attached to the segment $\De_a$ in the multisegment of $\oo$. We order the segments by increasing right endpoint, breaking ties by increasing left endpoint, as in \Cref{T:RoCon}. Write $\db^a$ for the dimension vector of $\De_a$. For the simultaneous filtration of this standard module we use the convention above, with exponent $i+\sum_{a<b}\langle\db^b,\db^a\rangle$ on the intertwiner from $\I(\pi)$ to the product in reverse order. The shift to the convention of \Cref{S:Jantzen} is
\[r_\pi=\sum_{a<b}\bigl(\dim_\CC\hom_Q(T(\De_a),T(\De_b))-\dim_\CC\hom_Q(T(\De_b),T(\De_a))\bigr),\]
where $T(\De_a)$ is the representation of the quiver $Q=(\cdots\to i\to i+1\to\cdots)$ attached to $\De_a$.
\begin{theorem}[Rogawski conjecture]
Let $\oo,\oo'\in\orb_\db$, $x\in\oo$ and $\pi=\Z(\oo)$. For a suitably generic deformation direction as in \Cref{T:RoCon}, the multiplicity of $\Z(\oo')$ in $\I(\pi)[i+r_\pi]$ is
\[\dim_\CC\HH^{i-\dim\oo}(\iota_x^*\ic{\oo'}),\]
where $\iota_x\colon\{x\}\hra\ep$.
\end{theorem}
Recall that $\HH^i(\iota_x^*\ic{\oo'})$ vanishes if $i\not\equiv\dim\oo'\pmod{2}$. Thus $\Z(\oo')$ appears only in degrees of a single parity in $\pi_1\grtimt\ldots\grtimt \pi_k$.

Finally, in the third part \Cref{S:charcy} we bound the pole order $\Lambda(\pi_1,\pi_2)$ and relate the image of the intertwiner to the conjecture of Lapid and Mínguez in \cite{LapMin25}. We first introduce some notation. We can identify the dual space $\ep^*$ with the space of representations of the opposite quiver of $A_\infty$ with dimension vector $\db$, \textit{i.e.} endomorphisms of degree $-1$ acting on $\bigoplus_{i\in\ZZ} \CC^{\db_i}$.
The cotangent space of $\ep$ is $T^*\ep=\ep\times \ep^*$,
with moment map
\[\mu\colon \ep\times \ep^*\ra \prod_{i\in\ZZ}\ed(\CC^{\db_i}), (T,T')\mapsto T\circ T'-T'\circ T.\] As in \cite{Lusztig1991Quivers} we define the Lagrangian subvariety
$\mathbf{\Lambda}(\db)\coloneq \mu^{-1}(0)$ and denote its set of irreducible components by $\comp_\db$.
For each $\oo\in\orb_\db$, we denote by
\[\mathbf{C}_\oo\coloneq \overline{T_\oo^*\ep}\] the closure of its conormal bundle.
    The set $\comp_\db$ of irreducible components of $\mathbf{\Lambda}(\db)$ is given by \[\comp_\db=\{\mathbf{C}_\oo:\oo\in\orb_\db\}.\]
    We recall that $\mathbf{\Lambda}(\db)$ parametrizes modules of the preprojective algebra $\Pi$ associated to the doubled quiver of $A_\infty$ with underlying graded vector space $\bigoplus_{i\in\ZZ}\CC^{\db_i}$.
    Now let $\mathbf{C}_1\in\comp_{\db^1}, \mathbf{C}_2\in\comp_{\db^2}$ and let $x_i\in \mathbf{C}_i$ be generic, \textit{i.e.}, such that the quantities  \[\dim_\CC\hom_\Pi(x_1,x_2),\,\dim_\CC\ext_\Pi^1(x_1,x_2)\] are minimized.
We denote their respective dimensions by
\[\hom(\mathbf{C}_1,\mathbf{C}_2),\, \ext^1(\mathbf{C}_1,\mathbf{C}_2).\] We also recall the binary operation of \cite{AizLap} given by generic extension
\[*\colon \comp_{\db^1}\times \comp_{\db^2}\ra\comp_{\db^1+\db^2}.\]
Finally, we write the singular support of $\ic{\oo}$ as $\siso(\ic{\oo})\subseteq \mathbf{\Lambda}(\db)$, \textit{cf.} \cite{KashiwaraSchapira1990Sheaves}. 
We have by \cite{Lusztig1991Quivers} that it is a union of irreducible components.
The main result of \Cref{S:charcy} is the following.
\begin{theorem}
Let $\pi_i\in\irr_{\db^i}$ correspond to $\oo_i\in\orb_{\db^i}$ and
write $\mathbf{C}_i=\mathbf{C}_{\oo_i}$. Then
\[\hom(\mathbf{C}_1,\mathbf{C}_2)-\ext^1(\mathbf{C}_1,\mathbf{C}_2)\le\Lambda(\pi_2,\pi_1)\le\hom(\mathbf{C}_1,\mathbf{C}_2).\]
If $\ext^1(\mathbf{C}_1,\mathbf{C}_2)=0$, both pole orders attain their respective
upper bounds and $J_{\pi_2,\pi_1}$ is an isomorphism.
If $\Lambda(\pi_2,\pi_1)=\hom(\mathbf{C}_1,\mathbf{C}_2)$, then every constituent
$\Z(\oo)$ of the image of $J_{\pi_2,\pi_1}$ satisfies
\[\mathbf{C}_2*\mathbf{C}_1\subseteq\siso(\ic{\oo}).\]
\end{theorem}
In \cite{Dro25} it is conjectured that the upper bound is always an equality.
Lapid and Mínguez conjecture in \cite{LapMin25} that the representation
corresponding to $\mathbf{C}_2*\mathbf{C}_1$ is a subrepresentation of $\pi_2\times\pi_1$.
When the upper bound is attained, the theorem identifies this representation
in the image if a constituent has irreducible singular support. In general
the singular support can have several components, as the example of
Kashiwara and Saito in \cite{KashiwaraSaito1997} shows.

Let us say a few words on the proof. We lift the construction $\hyp^+$ to Saito's theory of mixed Hodge modules. To a $\gdb$-equivariant mixed Hodge
module $M$ on $\ep$ one can associate a coherent sheaf on $\mathbf{\Lambda}(\db)$
denoted by $\gr^FM$, whose support is the singular support of the
underlying perverse sheaf. We apply the filtered comparison of
Appendix~\ref{A:filtered-hyp} directly to $\hyp^+M$ and to the opposite
chamber. For an intersection cohomology module, $\gr^FM$ is
Cohen--Macaulay. The dimensions of the spaces of extensions of two generic
$\Pi$-modules then bound the degrees in which the corresponding simple
perverse sheaf can occur in hyperbolic restriction.

\counterwithin{theorem}{section}
\counterwithin{lemma}{section}
\counterwithin{conjecture}{section}
\subsection*{AI disclosure}
Appendices A and B have been written with the assistance of AI (ChatGPT Astra 6) and the rest of the paper has been audited with the help of such systems. The author takes full responsibility for the correctness of the results in this paper.
These appendices were added firstly because the author intends to use their results in future work, secondly to streamline several arguments in previous versions, and finally, to close some gaps and fix some mistakes in previous versions.
\subsection*{Acknowledgements}
I want to thank Alberto Mínguez for his valuable feedback and encouragement to work on this problem. I also want to express my gratitude to Erez Lapid for the countless discussions we had on this and related problems, all of which were highly stimulating, and to Anton Mellit for allowing me to discuss several aspects of the paper with him. 
Finally, I want to express my gratitude towards Dan Ciubotaru, who pointed out a crucial gap in an early version of this paper regarding the associativity of the Jantzen filtration in \Cref{S:Jantzen,S:JanGeo}. The author was supported by a research grant
(VIL53023) from VILLUM FONDEN.
\section{Preliminaries}
In this section we establish the necessary background, first in the representation theory of $p$-adic groups and then in the algebraic geometry necessary for the Deligne--Langlands correspondence of \cite{KazhdanLusztig1987}, \cite{ChrissGinzburg1997}. 
\subsection{Representation theory}
We fix for the rest of the paper a non-archimedean local field $\FF$ with ring of integers $\mathrm{O}$, absolute value $\lvert-\rvert\colon \FF\ra \mathbb{R}_{\ge0}$ and residue field $k$ of cardinality $q$. We also choose a uniformizer $\varpi$ of $\FF$, \textit{i.e.}, a generator of the maximal ideal of $\mathrm{O}$.

We denote the category of smooth finite-length representations of $\bG(\FF)$ by $\rep(\bG)$ and the set of isomorphism classes of irreducible representations by $\irr(\bG)$. We set $\irr=\bigcup_{n\in\NN}\irr(\gl_n)$. We write for $\pi\in\irr_n\coloneq \irr(\gl_n)$ the degree $\deg(\pi)=n$. We also denote for $\pi\in\rep(\bG)$ by $[\pi]$ its image in the Grothendieck group of $\rep(\bG)$.
We recall the functor of (normalized) parabolic induction.
Let $\bP$ be a parabolic subgroup of $\bG$, $\bP=\bM\bU$ its Levi decomposition, $\delta_{\bP}$ its modular character, and $\op$ the opposite parabolic subgroup.
Define the exact functor
\[\id_\bP^\bG\colon \rep(\bM)\ra\rep(\bG),\]
\[(\pi,V)\mapsto \{f\colon\bG(\FF)\ra V : f\text{ locally constant},\ f(mug)=\pi(m)\delta_{\bP}^{1/2}(m)f(g)\},\] where $m\in\bM(\FF)$, $u\in\bU(\FF)$ and $g\in\bG(\FF)$.
Let $\alpha$ be a partition of $n$ and $\bP=\bP_{\alpha}$ the parabolic subgroup of $\gl_n$ containing the upper triangular matrices of $\gl_n$, denoted by $\bB_n$, and with Levi factor the block-diagonal matrices $\bM_\alpha=\gl_{\alpha_1}\times\ldots\times \gl_{\alpha_k}$ and unipotent part $\bU_\alpha$. As usual, we write for $\pi_1\otimes\ldots\otimes \pi_k\in\rep(\bM_\alpha)$
\[\pi_1\times\ldots\times\pi_k\coloneq\id_{\bP_\alpha}^{\gl_n}(\pi_1\otimes\ldots\otimes\pi_k),\, \pi_1\btimes\ldots\btimes\pi_k\coloneq \id_{\overline{\bP_\alpha}}^{\gl_n}(\pi_1\otimes\ldots\otimes\pi_k).\]
We call an irreducible representation cuspidal if it does not appear as a subrepresentation of a non-trivially induced representation. We denote the set of cuspidal representations of $\gl_n(\FF)$ by $\cus_n$ and define the equivalence relation $\rho\sim\rho'$ if there exists $k\in\ZZ$ such that $\rho\cong \lvert\det\rvert^k\otimes\rho'$. We also write $\cus=\bigcup_{n\in\NN}\cus_n$. We fix once and for all a representative of each class in $\cuss$ and by abuse of notation we denote this set of representatives by $\cuss$.
For every $\pi\in\irr(\gl_n)$, we can always find a partition $\alpha$ of $n$ and a cuspidal representation $\rho_1\otimes\ldots\otimes\rho_k$ of $\bM_\alpha(\FF)$ such that $\pi\hra\rho_1\times\ldots\times\rho_k$.
In this case the representation $\rho_1\otimes\ldots\otimes\rho_k$ is, up to permuting the indices, uniquely determined by $\pi$ and the formal sum
 \[\supp(\pi)\coloneq[\rho_1]+\ldots+[\rho_k]\] is called the cuspidal support of $\pi$.
\subsubsection{Intertwining operators}\label{S:int}
Let $\bP$ be a parabolic subgroup of $\gl_n$ with Levi factor $\bM=\gl_{n_1}\times\ldots\times \gl_{n_k}$, $\lambda$ a generic element in $\ZZ^k$, \textit{i.e.} all entries are pairwise different, and $\pi\in\rep(\bM)$.
For a character $\chi\colon\FF^\times\ra\CC^\times$ we write
\[\chi_\lambda\colon \bM\ra\CC^\times,\, (m_1,\ldots,m_k)\mapsto \prod_{i=1}^k\chi(\det(m_i))^{\lambda_i}.\]
Recall the Iwasawa decomposition
\[\gl_n(\FF)=\bP(\FF)K,\,K\coloneq\gl_n(\mathrm{O}).\]
    A function $f_s\in  \id_{\bP}^{\gl_n}(\pi\otimes{\lvert-\rvert_\lambda^s})$ is uniquely determined by the values it takes on $K$ and \[f_s\vert_K\in \id_{\bP(\FF)\cap K}^K(\pi\otimes\lvert-\rvert_\lambda^s\vert_{K\cap\bM(\FF)}),\]
where $\id_{\bP(\FF)\cap K}^K$ denotes the usual induction.
Now note that $\lvert-\rvert$ is trivial on $\det(K)$, hence the above equals
\[\id_{\bP(\FF)\cap K}^K(\pi\vert_{K\cap \bM(\FF)}).\]
Moreover, for $f\in \id_{\bP(\FF)\cap K}^K(\pi\vert_{K\cap\bM(\FF)})$, there is a unique extension \[f_s\in\id_{\bP}^{\gl_n}(\pi\otimes\lvert-\rvert_\lambda^s)\] via
\[f_s(muk)=\lvert m\rvert_\lambda^s\delta_{\bP}^{1/2}(m)\pi(m)f(k),\,m\in\bM(\FF),\,u\in\bU(\FF),\,k\in K.\]
We call this family the flat section through $f$. A holomorphic section germ through $f$ at $s=0$ is a finite sum of flat sections with holomorphic scalar coefficients near $0$, whose value at $0$ is $f$.

Following \cite[§ IV]{Waldspurger2003Plancherel}, see also \cite[§ 7]{Dat05}, one can define
for all $s\in\CC$ outside a discrete set a map
\[J_{\pi,\bP}(s)\colon\id_{\op}^{\gl_n}(\pi\otimes\lvert-\rvert_\lambda^s)\ra \id_{\bP}^{\gl_n}(\pi\otimes\lvert-\rvert_\lambda^s)\]
which is uniquely defined by the property that for $f$ with compact support in $\op(\FF)\bP(\FF)$ modulo $\op(\FF)$
\[J_{\pi,\bP}(s)(f)(1_{\gl_n})=\int_{\bU(\FF)}f(u)\mathrm{d}u.\]
Note that this depends on the choice of a Haar measure on $\FF$, which we fix for the rest of the paper.
\begin{theorem}[{\cite[§ IV]{Waldspurger2003Plancherel}, \cite[§ 7]{Dat05}}]
    The morphism $J_{\pi,\bP}(s)$ can be meromorphically continued to a morphism for all $s$, \textit{i.e.}, there exists $\Lambda(\pi,\bP)\in\ZZ_{\ge0}$ such that the map
    \[J_{\pi,\bP} \colon \id_{\op}^{\gl_n}(\pi)\ra \id_{\bP}^{\gl_n}(\pi),\, f\mapsto \lim_{s\ra 0}s^{\Lambda{(\pi,\bP)}}J_{\pi,\bP}(s)(f_s),\] where $\{f_s\in \id_{\op}^G(\pi\otimes\lvert-\rvert_\lambda^s)\}_{s\in \CC}$ is a holomorphic family with $f=f_0$, is well-defined and non-zero.
\end{theorem}
Let $\alpha=(\alpha_1,\ldots,\alpha_k)$ be a partition of $n$, $\bP=\bP_\alpha$ a standard parabolic subgroup, and $\pi_i\in\rep(\gl_{\alpha_i})$. We write \[\Lambda(\pi_1,\ldots,\pi_k)\coloneq \Lambda(\pi_1\otimes\ldots\otimes \pi_k,\bP_{\alpha}),\quad\bLambda(\pi_1,\ldots,\pi_k)\coloneq \Lambda(\pi_1\otimes\ldots\otimes\pi_k,\overline{\bP_{\alpha}}) \] as well as the morphisms \[J_{\pi_1,\ldots,\pi_k}\coloneq J_{\pi_1\otimes\ldots\otimes\pi_k,{\bP_{\alpha}}},\, \bJ_{\pi_1,\ldots,\pi_k}\coloneq J_{\pi_1\otimes\ldots\otimes\pi_k,\overline{\bP_{\alpha}}}.\]
We use a similar notation for their versions in families $J_{\pi_1\otimes\ldots\otimes\pi_k,{\bP_{\alpha}}}(s)$.
If moreover each $\pi_i$ is induced from irreducible representations, we can define $\alpha(\pi_1,\ldots,\pi_k)$ as the unique integer such that 
\[J_{\pi_1\otimes\ldots\otimes\pi_k,{\bP_{\alpha}}}(s)\circ J_{\pi_1\otimes\ldots\otimes\pi_k,\overline{\bP_{\alpha}}}(s)=\xi(s) s^{-\alpha(\pi_1,\ldots,\pi_k)},\,\xi(s)\in\CC[[s]]^\times\] for almost all $s\in \CC$, see \cite[§ IV]{Waldspurger2003Plancherel}, \cite[§ 7.6]{Dat05}, or \cite[§4]{Dro25}.
Indeed, if all $\pi_i$ are irreducible, this follows from the generic irreducibility of the parabolic induction as explained in the above sources. For induced representations the claim follows from the claim for irreducible representations together with \cite[Proposition 7.8]{Dat05}.
Finally, we set
\[\pard(\pi_1,\ldots,\pi_k)\coloneq \Lambda(\pi_1,\ldots,\pi_k)+\bLambda(\pi_1,\ldots,\pi_k)-\alpha(\pi_1,\ldots,\pi_k).\]
\begin{lemma}[{\cite[Lemma 2.3]{LapMin25}, \cite{KKKO}}]\label{L:intiso}
    We have that $\pard(\pi_1,\ldots,\pi_k)\ge 0$ and
    the following are equivalent.
    \begin{enumerate}
        \item $J_{\pi_1,\ldots,\pi_k}$ is an isomorphism.
        \item $J_{\pi_k,\ldots,\pi_1}$ is an isomorphism.
        \item $\pard(\pi_1,\ldots,\pi_k)=0$.
    \end{enumerate}
\end{lemma}
Let $\overline\alpha=(\alpha_k,\ldots,\alpha_1)$, and let $w_\alpha$ be the permutation matrix which sends the ordered blocks of sizes $\alpha_1,\ldots,\alpha_k$ to the reverse order, preserving the order within each block. Thus $w_{\overline\alpha}=w_\alpha^{-1}$. Reversing the tensor factors, the map $\mathrm{Ad}(w_\alpha)(f)(g)=f(w_\alpha^{-1}g)$ identifies induction from $\bP_\alpha$ with induction from $\overline{\bP_{\overline\alpha}}$, and similarly for the opposite parabolics.
Then
\[J_{\pi_1,\ldots,\pi_k}\circ\mathrm{Ad}(w_{\overline\alpha})=\mathrm{Ad}(w_{\overline\alpha})\circ \bJ_{\pi_k,\ldots,\pi_1}\]
and similarly for the families in $s$.

In the next sections we recall the relevant notions from algebraic geometry.
\subsection{Constructible sheaves}\label{S:const}
We start by setting the stage for the constructible side of the local Langlands correspondence. For a more in-depth discussion of the mentioned notions we refer the reader to \cite{BBD1982}, \cite{ChrissGinzburg1997}, or \cite{Achar2021Perverse}. 

Let $\bX$ be a complex variety and $\bG$ a complex linear algebraic group acting on $\bX$.
We denote by $D_c^b(\bX,\bG)$ the $\bG(\CC)$-equivariant bounded derived category of constructible sheaves with complex coefficients on $\bX(\CC)$. If we want to highlight that we consider a different field of coefficients $R$, we write $D_c^b(\bX,\bG,R)$.
We denote by $\DD$ Verdier duality, $[-]$ the shift functor, and if $f\colon \bX\ra \bY$ is a morphism of $\bG$-varieties we denote as usual the corresponding functors
\[f_*,f_!\colon D_c^b(\bX,\bG)\ra D_c^b(\bY,\bG),\, f^*,f^!\colon D_c^b(\bY,\bG)\ra D_c^b(\bX,\bG).\]
If $\bH$ is a closed subgroup of $\bG$, we denote the forgetful functor by \[\mathrm{For}_{\bG}^{\bH}\colon D_c^b(\bX,\bG)\ra D_c^b(\bX,\bH).\] On the other hand, if $\bH$ is a normal closed subgroup of $\bG$ acting trivially on $\bX$, we denote the inflation functor by
\[\mathrm{Inf}_{\bG/\bH}^{\bG}\colon D_c^b(\bX,\bG/\bH)\ra D_c^b(\bX,\bG).\]
In this case we also recall the functors of invariants and co-invariants
\[\mathrm{Inv}_{\bH,*}\colon D_c^b(\bX,\bG)\ra D_c^+(\bX,\bG/\bH),\, \mathrm{Inv}_{\bH,!}\colon D_c^b(\bX,\bG)\ra D_c^-(\bX,\bG/\bH).\]
Note that these objects do \emph{a priori} no longer land in the bounded derived category, but rather in a bounded below or bounded above version. In the situations considered below, we will see that they land in the bounded derived category.

For $\cf$ in $D_c^b(*,\gm)$, we set $\HH_{\gm}^k(\cf)\coloneq \hom_{D_c^b(*,\gm)}(\underline{\CC}_*,\cf[k])$. 
Recall that \[\HH_{\gm}^\bullet(\uc_*)\cong \CC[s]\] with $s$ in degree $2$ and hence $\HH_\gm^\bullet(\cf)$ has a natural action of $\CC[s]$. Given a complex $\gm$-variety $\bX$, let $a_{\bX}\colon\bX\ra B\gm$ be the projection to a point and set \[\tH_{BM,\gm}^k(\bX)\coloneq \HH_{\gm}^{-k}(\DD(\uc_{\bX})).\] We will write $\widehat{\tH}_{BM,\gm}^\bullet(\bX)$, and similarly $\widehat{\HH}_{\gm}^\bullet(\cf)$, for the direct products. Then $\widehat{\tH}_{BM,\gm}^\bullet(*)=\CC[[s]]$.

We write $\perv(\bX,\bG)$ for the category of $\bG$-equivariant perverse sheaves and ${}^p\!\tH^i$ for the $i$th perverse cohomology functor. 
If $\bG$ has finitely many orbits on $\bX$ and the stabilizer of every $x\in\bX$ is connected, then the simple perverse sheaves are the intersection cohomology complexes $\ic{\oo}$, where $\oo$ ranges over the $\bG$-orbits of $\bX$. We will assume this from now on and recall that for two $\bG$-orbits $\oo,\oo'$ and $k\in\ZZ_{<0}$
\[\dim_\CC\hom(\ic{\oo},\ic{\oo'})=\delta_{\oo,\oo'},\, \ext^k(\ic{\oo},\ic{\oo'})=0,\]
where the Hom- and Ext-spaces can be taken either in $D_c^b(\bX,\bG)$ or $D_c^b(\bX)$.
Let $\sem(\bX,\bG)^+$ be the monoid generated under $\oplus$ by $\ic{\oo}[k]$, $k\in \ZZ$, $\oo$ a $\bG$-orbit and 
$\sem(\bX,\bG)$ the group obtained by formal inverses. Note that $\sem(\bX,\bG)$ admits an action of $\ZZ[t,t^{-1}]$ via
\[t\cdot\cf=\cf[-1].\]
We denote by $\overline{(-)}\colon \sem(\bX,\bG)\ra\sem(\bX,\bG)$ the map $\cf\mapsto \DD(\cf)$.
\subsection{Mixed Hodge modules}\label{S:MHM}
We give a brief and superficial recollection of the relevant properties of mixed Hodge modules of \cite{Saito1988}, \cite{Saito1990}. We refer the reader to \cite{Schnell2014OverviewMHM} or \cite{Sabbah2019IntroMHM} for a more detailed exposition. In particular, we will use the language of ($\bG$-equivariant) $\mathscr{D}$-modules freely.
Let $\bX$ and $\bG$ be as in \Cref{S:const}. Recall the Riemann--Hilbert correspondence \[\mathrm{DR}\colon \{\text{regular, holonomic }\bG\text{-equivariant }\mathscr{D}\text{-modules on }\bX\}\ra \perv(\bX,\bG,\CC).\]

We let $\mhm(\bX,\bG)$ be the category of mixed Hodge modules (over $\QQ$).
An object in the category $\mhm(\bX,\bG)$ consists of a quadruple $M=(\mathcal{M},F_\bullet, \cf,W_\bullet)$, where
\begin{enumerate}
    \item $\mathcal{M}$ is a left regular, holonomic (strongly) $\bG$-equivariant $\mathscr{D}$-module.
    \item $F_\bullet$ is a good filtration of $\mathcal{M}$.
    \item $\cf\in\perv(\bX,\bG,\QQ)$ such that $\cf\otimes_\QQ\CC\cong \mathrm{DR}(\mathcal{M})$.
    \item $W_\bullet$ is an increasing filtration.
\end{enumerate}
We call a mixed Hodge module $M$ pure of weight $w$ if $W_k=0$ for $k<w$ and $W_k=\cf$ for $k\ge w$.
The derived category of mixed Hodge modules is denoted by $D\mhm(\bX,\bG)$ and an object $M$ of $D\mhm(\bX,\bG)$ is called pure of weight $w$ if for all $i\in\ZZ$, the cohomology group $\tH^i(M)$ is pure of weight $w+i$.
We let \[\rat\colon D\mhm(\bX,\bG)\ra D_c^b(\bX,\bG,\QQ)\] denote the rational realization functor, which sends $(\mathcal M,F_\bullet,\cf,W_\bullet)$ to its underlying rational perverse sheaf $\cf$.
One of the many reasons why mixed Hodge modules are such a powerful tool is the following theorem, whose ramifications are hard to overstate.
\begin{theorem}[\cite{deligne1971theorie2}, \cite{griffiths1968periods1}, \cite{Saito1990}]\label{T:decomp}
    If $M$ is a pure object in $D\mhm(\bX,\bG)$ then
    \[M\cong \bigoplus_{i\in\ZZ}\tH^i(M)[-i].\]
    If $M$ is a pure mixed Hodge module, then $\rat(M)$ is a direct sum of simple perverse sheaves.
\end{theorem}
The functors of \Cref{S:const} can all be lifted to $D\mhm(\bX,\bG)$ such that under $\rat$, they correspond to the functors of \Cref{S:const}.
Moreover, any $\ic{\oo}$ can be lifted to a pure mixed Hodge module of weight $\dim\oo$, which we denote by abuse of notation also $\ic{\oo}$.

Next, let $T^*\bX$ be the cotangent bundle of $\bX$, which carries a natural $\bG$-action coming from the $\bG$-action on $\bX$. Then for $M$ a mixed Hodge module, $\gr^F M\coloneq \bigoplus_{i\in\ZZ}F_{i-1}\backslash F_i$ naturally has the structure of a $\bG$-equivariant sheaf on $T^*\bX$. We recall that if $M=\ic{\oo}$ and $\mathbf{U}$ is a generic open subset of the conormal bundle $T_\oo^*\bX$, then 
    \[\gr^F(\ic{\oo})\lvert_{\mathbf{U}}=\mathscr{O}_{\mathbf{U}}.\]
    Finally, define the singular support of $M$ as
    \[\siso(M)\coloneq \mathrm{supp}(\gr^FM)\subseteq T^*\bX.\]
\subsection{Quiver varieties}\label{S:quvar}
We now come to the specific varieties that will govern the representation theory of $\gl_n(\FF)$. We will follow the presentation of \cite{LapMin25}.

Let $Q$ be the quiver $A_\infty$ with an arrow $i\mapsto i+1$ for all $i\in\ZZ$. We let \[\grdim(Q)\coloneq \bigoplus_{i\in\ZZ}\NN\] and write for $\db\in\grdim(Q)$, $\lvert\db\rvert\coloneq \sum_{i\in\ZZ}\db_i$. Denote for $\db\in\grdim(Q)$ by $\ep$ the complex variety of representations of $Q$ with underlying vector space \[\CC^\db\coloneq \bigoplus_{i\in\ZZ}\CC^{\db_i},\] 
\[\ep(\CC)\coloneq \{(T_i)_{i\in\ZZ}:T_i\colon \CC^{\db_i}\ra\CC^{\db_{i+1}}\,\CC\text{-linear}\}.\] We set \[\gdb(\CC)\coloneq \bigtimes_{i\in\ZZ}\gl_{\db_i}(\CC)\] and let it act on $\ep$ by conjugation. We recall the classification of the set $\orb_\db$ of $\gdb$-orbits of $\ep$.

For integers $a\le b$, denote the corresponding segment by $[a,b]$. If $a>b$, we interpret $[a,b]$ as the empty segment. We define its length $l([a,b])\coloneq\max\{0,b-a+1\}$ and 
its dimension vector as \[\grdim([a,b])\in\grdim(Q),\, \grdim([a,b])_i\coloneq\begin{cases}
    1&\text{if }a\le i\le b,\\0&\text{otherwise}.
\end{cases}\]
A multisegment $\fm$ is defined to be a formal sum of segments and we extend lengths and dimension vectors additively to the set of multisegments $\mul$. For $\db\in\grdim(Q)$ we let \[\mul_\db=\{\fm\in\mul:\grdim(\fm)=\db\}\] and define for $\fm\in\mul_\db$ a representation $T(\fm)\in \ep$ as follows.
If $\fm=[a,b]$ is a single segment, we define the underlying linear maps \[T_i=\begin{cases}
    \mathbbm{1}_\CC& \text{if }a\le i\le b-1\\0&\text{otherwise}.
\end{cases}\]
For $\fm=\De_1+\ldots+\De_k$ a sum of segments we simply set $T(\fm)\coloneq T(\De_1)\oplus\ldots\oplus T(\De_k)$. Note that this depends on the order of the segments.
We then write
\[\oo\colon \mul_\db\ra\orb_\db,\, \oo(\fm)\coloneq \gdb\cdot T(\fm).\]
\begin{theorem}[Gabriel's theorem, \cite{Gabriel1972}]
    The map $\oo\colon \mul_\db\ra\orb_\db$ is a bijection.
\end{theorem}
We say a segment $[a,b]$ precedes a segment $[a',b']$ if $a+1\le a'\le b+1\le b'$, in which case we write $[a,b]\prec [a',b']$. 
In this case the union and intersection are defined as \[[a,b]\cup[a',b']\coloneq [a,b'],\, [a,b]\cap[a',b']\coloneq [a',b].\]
Segments $[a,b]$ and $[a',b']$ are called unlinked if they do not precede each other.
Finally, we write $\fn\le \fm$ if $\fn$ can be obtained from $\fm$ via finitely many elementary operations, \textit{i.e.}, the replacement of two linked segments by their intersection and union.
\begin{theorem}[{\cite[Theorem 2.2]{Zelevinsky1981}}]\label{T:Oqui}
    Let $\fm,\fn\in\mul_\db$. Then $\fn\le\fm$ if and only if $\oo(\fm)\subseteq\overline{\oo(\fn)}$.
\end{theorem}
\subsection{Classification of irreducible representations}\label{S:class}
Next we recall from \cite{Zel80} the classification of $\irr$.
For $\rho\in \cuss$ we let $\mul_\rho$ be the set of multisegments where we decorate each segment with a $\rho$, \emph{e.g.} $[a,b]_\rho$, and we have a natural bijection $\mul\ra\mul_\rho$ sending $\fm$ to $\fm_\rho$.
We let $\mul_\cus$ be the set of finite sums $\fm_{1}+\ldots+\fm_{k},\, \fm_i\in \mul_{\rho_i}$ for some $\rho_i\in\cuss$. We define \[\cusp([a,b]_\rho)\coloneq [\rho\otimes \lvert\det\rvert^a]+\ldots+[\rho\otimes \lvert\det\rvert^b]\] and extend this additively to $\mul_\cus$. Finally, for two segments $\De_1,\De_2\in\mul_{\cus}$, we write $\De_1\prec\De_2$ if $\De_1,\De_2\in\mul_{\rho}$ and $\De_1$ precedes $\De_2$ in the sense of \Cref{S:quvar}. An elementary operation on a multisegment is again an elementary operation in the sense of \Cref{S:quvar}, where the two segments involved have to be contained in $\mul_{\rho}$ for some $\rho\in\cuss$. Finally, we write $\fn\le\fm$ if $\fn$ can be obtained from $\fm$ via finitely many elementary operations.
\begin{lemma}[{\cite[§ 3]{Zel80}}]\label{L:seg}
    Let $[a,b]_\rho$ be a segment in $\mul_\cus$. Then there exists a unique irreducible subrepresentation
    $\Z([a,b]_\rho)$ of $\rho\otimes\lvert-\rvert^a\times\ldots\times\rho\otimes\lvert-\rvert^b$.
\end{lemma}
\begin{theorem}[{\cite[§6]{Zel80}}]\label{T:zel}
    There exists a bijection 
    \[\LL\colon \mul_\cus\ra\irr\] satisfying the following properties.
    \begin{enumerate}
        \item $\cusp(\Z(\fm))=\cusp(\fm)$.
        \item If $\fm_1\in \mul_{\rho_1},\ldots,\fm_k\in \mul_{\rho_k},\,\fm=\fm_1+\ldots+\fm_k$ for pairwise non-equivalent cuspidal representations $\rho_i$, then
        \[\Z(\fm)\cong \Z(\fm_1)\times\ldots\times\Z(\fm_k).\]
        \item If $\De_1,\ldots,\De_k$ are segments which are pairwise unlinked, the representation
        \[\Z(\De_1)\times\ldots\times \Z(\De_k)\] is irreducible.
        \item Assume that $\fm=\De_1+\ldots+\De_k$ such that if $i>j$, then $\De_i\nprec \De_j$. Then $\Z(\fm)$ is the unique irreducible quotient of 
        \[\I(\fm)\coloneq \Z(\De_1)\times\ldots\times\Z(\De_k)\]
        and appears with multiplicity one in its composition series.
    \end{enumerate}
\end{theorem}
We call a representation of the form $\I(\fm)$ a standard module and note that by (3), $\I(\fm)$ is up to isomorphism uniquely defined. Indeed, it is easy to see that any two admissible orderings of (4), which from now on will be called \emph{arranged} orderings, can be transformed into each other by switching two neighboring unlinked segments. 
We will write $\I(\pi)$ for $\I(\fm)$ if $\pi\cong\Z(\fm)$.
As a consequence of the above theorem, we can write any $\pi\in\irr$ as
\[\pi\cong\bigtimes_{\rho\in\cuss}\pi_\rho,\]
where $\pi_\rho=\Z(\fm)$ for some $\fm\in\mul_\rho$ and for almost all $\rho$, $\deg(\pi_\rho)=0$.
For $\pi\in\irr$ and $\rho\in\cuss$, let $\grdim(\pi)_\rho\in\grdim(Q)$ be the dimension vector whose $j$th entry is the multiplicity of $\rho\otimes\lvert\det\rvert^j$ in $\supp(\pi_\rho)$. For $\db\in\grdim(Q)$ we let
\[\irr_{\db,\rho}\coloneq\{\pi\in\irr:\pi\cong\pi_\rho,\,\grdim(\pi)_\rho=\db\}.\]
There exists thus a bijection \[\LL_\rho\colon \orb_\db\ra\irr_{\db,\rho},\,\oo(\fm)\mapsto \Z(\fm_\rho).\]
For $\db^1,\db^2\in\grdim(Q)$ write \[\langle\db^1,\db^2\rangle=\sum_{i\in\ZZ}\db_{i}^1\db_{i}^2-\db_{i}^1\db_{i+1}^2,\,(\db^1,\db^2)=\langle\db^1,\db^2\rangle+\langle\db^2,\db^1\rangle\] and recall, \textit{cf.} \cite[§2]{AizLap}, that for all $x_i\in\mathbf{E}_{\db^i}, i=1,2$
\[\dim_\CC\hom_{Q}(x_1,x_2)-\dim_\CC\ext_Q^1(x_1,x_2)=\langle\db^1,\db^2\rangle.\]
We also define 
\[\alpha_+(x_1,x_2)\coloneq \dim_\CC\hom_{Q}(x_1,x_2)-\dim_\CC\ext^1_{Q}(x_2,x_1)\]
and hence
\[\alpha_+(x_1,x_2)+\alpha_+(x_2,x_1)=(\db^1,\db^2).\]
For future reference we recall that
\[\dim_\CC\hom_{Q}(T([a,b]),T([c,d]))=\begin{cases}
    1& \text{if }c\le a\le d\le b,\\0&\text{otherwise},
\end{cases}\]
\[\dim_\CC\ext^1_{Q}(T([a,b]),T([c,d]))=\begin{cases}
    1& \text{if }a+1\le c\le b+1\le d,\\0&\text{otherwise}.
\end{cases}\]
Now let $\pi,\pi'\in \irr$. We set
\[\langle\pi,\pi'\rangle\coloneq \sum_{\rho\in\cuss}\langle\grdim(\pi_\rho)_\rho,\grdim(\pi_\rho')_\rho\rangle,\, (\pi_1,\pi_2)\coloneq \langle\pi_1,\pi_2\rangle+\langle\pi_2,\pi_1\rangle.\]
For $\rho\in\cuss$, we let $T_\rho(\pi)=T(\fm)$, where $\fm\in\mul$ is such that $\pi_\rho\cong\Z(\fm_\rho)$.
We then set \[\alpha_+(\pi,\pi')\coloneq\sum_{\rho\in\cuss}\alpha_+(T_\rho(\pi),T_\rho(\pi')).\]
Note that \[\alpha_+(\pi,\pi')+\alpha_+(\pi',\pi)=(\pi,\pi').\]
More generally, if $\pi=\pi_1\times\ldots\times\pi_k,\,\pi_i=\Z(\fm_i),\fm_i\in\mul_\cus$, we let $\Tilde{\pi}=\Z(\fm_1+\ldots+\fm_k)$ and thus set for $\pi,\pi'$ induced from irreducible representations
\[\alpha_+(\pi,\pi')\coloneq \alpha_+(\Tilde{\pi},\Tilde{\pi'}).\]
Note that $\langle-,-\rangle$ can be extended similarly to representations induced from irreducible representations.
For $\pi_1,\ldots,\pi_k$ irreducible or induced from irreducible representations, we write \[\alpha_+(\pi_1,\ldots,\pi_k)\coloneq\sum_{i<j}\alpha_+(\pi_i,\pi_j).\]
\section{Generalized Jantzen filtration}\label{S:Jantzen}
In \cite{Rogawski1985}, the author introduces a filtration on standard modules coming from the analytic behavior of the intertwining operators and conjectures that it is governed by the stalks of perverse sheaves on $\ep$. This conjecture is the analogue of the Jantzen filtration on Verma modules \cite{Jantzen1979}. 
In this section we generalize these constructions to arbitrary induced representations and prove some useful properties that will allow us to compute the associated graded representation in terms of the geometry of the quiver varieties in \Cref{S:JanGeo}.
\begin{lemma}\label{L:alpha}
    Let each $\pi_i$, $1\le i\le k$, be irreducible or induced from irreducible representations. Then
 \[\sum_{i< j}(\pi_i,\pi_j) =\alpha(\pi_1,\ldots,\pi_k).\]
 In particular, $\alpha$ only depends on the cuspidal support of the representations involved.
\end{lemma}
\begin{proof}
    Embed all $\pi_i\hra \rho_{i,1}\times\ldots\times \rho_{i,a_i}$ where all $\rho_{i,j}$ are cuspidal. By standard properties of intertwining operators, see for example \cite[§4]{Dro25}, we have that
    \[\alpha(\pi_1,\ldots,\pi_k)=\alpha(\rho_{1,1}\times\ldots\times \rho_{1,a_1},\ldots,\rho_{k,1}\times\ldots\times \rho_{k,a_k}),\]
    which is the same as
    \[\sum_{1\le i<j\le k}\sum_{a=1}^{a_i}\sum_{b=1}^{a_j}\alpha(\rho_{i,a},\rho_{j,b}),\]
    \textit{cf.} \cite[§4]{Dro25}.
    It is well known, see for example \cite[7.5, 8.4]{Dat05} or \cite{FinisLapidMuller2012}, that for any two cuspidal representations $\rho,\rho'$ we have
    \[\alpha(\rho,\rho')=\begin{cases}
        2&\text{if }\rho\cong\rho',\\
        -1&\text{if }\rho\cong\rho'\lvert-\rvert^{\pm1},\\
        0&\text{otherwise}.
    \end{cases}\]
    The claim then follows from the definitions of \Cref{S:class}.
\end{proof}
Let $\pi_1,\ldots,\pi_k\in\irr$ and fix a generic element $\lambda\in\ZZ^k$ and recall the intertwining operator
\[J_{\pi_1,\ldots,\pi_k}(s)\colon \pi_1\lvert-\rvert^{\lambda_1s}\btimes\ldots\btimes\pi_k\lvert-\rvert^{\lambda_ks}\ra \pi_1\lvert-\rvert^{\lambda_1s}\times\ldots\times\pi_k\lvert-\rvert^{\lambda_ks}\] of \Cref{S:int}.
If $f\in\pi_1\btimes\ldots\btimes\pi_k$ or $\pi_1\times\ldots\times\pi_k$, we let $f_s$ range over holomorphic sections with $f_0=f$ in the compact realization of \Cref{S:int}. In the definitions below, the indicated limit is required to exist for some such $f_s$.

We will now introduce an increasing filtration $\{\bV_i(\pi_1,\ldots,\pi_k)\}_{i\in\ZZ}$ of $\pi_1\btimes\ldots\btimes\pi_k$
\[\bV_i(\pi_1,\ldots,\pi_k)\coloneq \{f\in \pi_1\btimes\ldots\btimes\pi_k: \lim_{s\ra 0}s^{i+\alpha_+(\pi_k,\ldots,\pi_1)}J_{\pi_1,\ldots,\pi_k}(s)(f_s)\text{ exists}\}.\]
Similarly, we define a filtration $\{V_i(\pi_1,\ldots,\pi_k)\}_{i\in\ZZ}$ of $\pi_1\times\ldots\times\pi_k$
\[V_i(\pi_1,\ldots,\pi_k)\coloneq \{f\in \pi_1\times\ldots\times\pi_k: \lim_{s\ra 0}s^{i+\alpha_+(\pi_1,\ldots,\pi_k)}\bJ_{\pi_1,\ldots,\pi_k}(s)(f_s)\text{ exists}\}.\]
\begin{lemma}\label{L:filtpol}
    For $i\ll 0$, $V_i(\pi_1,\ldots,\pi_k)=\bV_i(\pi_1,\ldots,\pi_k)=0$. Moreover, $\Lambda(\pi_1,\ldots,\pi_k)$ is the minimal integer $i$ such that $\bV_{i-\alpha_+(\pi_k,\ldots,\pi_1)}=\pi_1\btimes\ldots\btimes\pi_k$ and similarly for $\bLambda(\pi_1,\ldots,\pi_k)$.
\end{lemma}
\begin{proof}
The second claim follows from the definition of the order of the pole, since
    after clearing the pole, the value at $s=0$ depends only on $f_0$
    and is nonzero for some $f_0$. For the first claim, note that on flat sections the intertwining operator is rational in $q^s$
    by \cite[§IV]{Waldspurger2003Plancherel}. Its inverse is meromorphic
    by the composition identity of \Cref{S:int}. Thus for each
    nonzero $f$ there is a least $i$ such that
    $f\in\bV_i(\pi_1,\ldots,\pi_k)$. Indeed, applying this inverse to
    the holomorphic family
    $s^{i+\alpha_+(\pi_k,\ldots,\pi_1)}J_{\pi_1,\ldots,\pi_k}(s)(f_s)$
    bounds the pole order of $s^{i+\alpha_+(\pi_k,\ldots,\pi_1)}f_s$.
    Since $\pi_1\btimes\ldots\btimes\pi_k$ is of finite length,
    the claim follows. The same argument applies to $V_i$.
\end{proof}

We define \[\pi_1\btimes\ldots\btimes\pi_k[i]\coloneq \bV_{i-1}(\pi_1,\ldots,\pi_k)\backslash \bV_i(\pi_1,\ldots,\pi_k)\] and similarly $\pi_1\times\ldots\times\pi_k[i]$.
Note that due to the last comment of \Cref{S:int}, \[\mathrm{Ad}(w_\alpha)(\bV_i(\pi_1,\ldots,\pi_k))=V_i(\pi_k,\ldots,\pi_1)\]
for all $i$ and hence
\[\pi_1\btimes\ldots\btimes\pi_k[i]\cong \pi_k\times\ldots\times\pi_1[i].\]
We observe that if the inducing representations are replaced by
subrepresentations or quotients, the induced maps preserve the condition
that the limit exists for a given power of $s$, by naturality of the
intertwining operators, \textit{cf.} \cite[Lemma 2.3]{lapid2018geometric}.

We continue with the following technical lemma.
\begin{lemma}\label{L:ind}
    Assume $k=2$, and let $\lambda=(\lambda_1,\lambda_2)$ with $\lambda_1\neq \lambda_2$. 
   Denote by
    \[J_{\pi_1,\pi_2,\lambda_1,\lambda_2}(s)\colon \pi_1\lvert-\rvert^{\lambda_1s}\times \pi_2\lvert-\rvert^{\lambda_2s}\ra \pi_2\lvert-\rvert^{\lambda_2s}\times \pi_1\lvert-\rvert^{\lambda_1s}\]the map $J_{\pi_2,\pi_1}(s)\circ \mathrm{Ad}(w_{\deg(\pi_1),\deg(\pi_2)})$.
    Then 
    \[V_{i,\lambda_1,\lambda_2}(\pi_1,\pi_2)\coloneq \{f\in\pi_1\times\pi_2:\lim_{s\ra 0}s^{i+\alpha_+(\pi_1,\pi_2)}J_{\pi_1,\pi_2,\lambda_1,\lambda_2}(s)(f_s)\text{ exists}\}\]
    is independent of $\lambda_1$ and $\lambda_2$.
\end{lemma}
\begin{proof}
        We first note that
        \[J_{\pi_1,\pi_2,\lambda_1,\lambda_2}(s)=J_{\pi_1,\pi_2,\lambda_1-\lambda_2,0}(s)\otimes \lvert\det\rvert^{\lambda_2s}.\] Indeed, it suffices to check this for all but finitely many $s$, where one can use the defining property of the intertwining operators of \Cref{S:int}, see also \cite[Lemma 7.12]{Dat05}.
        Hence one has
        \[V_{i,\lambda_1,\lambda_2}(\pi_1,\pi_2)=V_{i,\lambda_1-\lambda_2,0}(\pi_1,\pi_2).\]
        Next, we can replace $s$ by $as$ for some $a\in\CC^\times$, and this will not affect $V_{i,\lambda_1,\lambda_2}(\pi_1,\pi_2)$. But one has that
        \[J_{\pi_1,\pi_2,\lambda_1-\lambda_2,0}(as)=J_{\pi_1,\pi_2,a(\lambda_1-\lambda_2),0}(s).\]
        Since for any other generic $\lambda'=(\lambda_1',\lambda_2')$ one can find $a\in\CC^\times$ with $a(\lambda_1-\lambda_2)=\lambda_1'-\lambda_2'$, the claim follows.
\end{proof}
\begin{lemma}\label{L:sym}
    Assume all $\pi_i$ are irreducible. Then the map
    \[T_i(\pi_1,\ldots,\pi_k)\colon\pi_1\btimes\ldots\btimes\pi_k[i]\ra \pi_1\times\ldots\times\pi_k[-i],\, [f]\mapsto [\lim_{s\ra 0}s^{i+\alpha_+(\pi_k,\ldots,\pi_1)}J_{\pi_1,\ldots,\pi_k}(s)(f_s)]\]
    is an isomorphism, where $f_s$ is any section for which the displayed limit exists.
\end{lemma}
\begin{proof}
For any holomorphic section $f_s$ for which the indicated limit exists,
    \[s^{i+\alpha_+(\pi_k,\ldots,\pi_1)}J_{\pi_1,\ldots,\pi_k}(s)(f_s)\]
    extends holomorphically to $s=0$. In order to show
    that the above map is well defined, we recall from \Cref{S:int}
    and \Cref{L:alpha} that
    \[s^{-i+\alpha_+(\pi_1,\ldots,\pi_k)}\bJ_{\pi_1,\ldots,\pi_k}(s)s^{i+\alpha_+(\pi_k,\ldots,\pi_1)}J_{\pi_1,\ldots,\pi_k}(s)(f_s)=\xi(s) f_s,\,\xi(s)\in\CC\{s\}^\times.\]
    Thus the value of this family at $s=0$ belongs to
    $V_{-i}(\pi_1,\ldots,\pi_k)$. If two such sections have the same
    value $f$ at $s=0$, their difference is divisible by $s$. Dividing
    the displayed identity for their difference by $s$ shows that the
    two limits agree modulo $V_{-i-1}(\pi_1,\ldots,\pi_k)$. Moreover,
    $T_i(\pi_1,\ldots,\pi_k)$ vanishes on
    $\bV_{i-1}(\pi_1,\ldots,\pi_k)$, since for such a vector we may
    choose $f_s$ for which the limit already exists at $i-1$.
    By the same identity, the analogous map in the reverse direction
    composes with $T_i(\pi_1,\ldots,\pi_k)$ in either order to
    multiplication by $\xi(0)$. Hence $T_i(\pi_1,\ldots,\pi_k)$ is
    an isomorphism.
\end{proof}
We now define a $\ZZ[t,t^{-1}]$-linear multiplication on $\ZZ[\irr][t,t^{-1}]$ via
\[\grtim\colon \ZZ[\irr][t,t^{-1}]\times \ZZ[\irr][t,t^{-1}]\ra \ZZ[\irr][t,t^{-1}],\, (\pi_1,\pi_2)\mapsto \sum_{i\in\ZZ}[\pi_1\times\pi_2[i]]_{ss}t^i,\]
where $[\pi_1\times\pi_2[i]]_{ss}$ denotes the semi-simplification of $\pi_1\times\pi_2[i]$.
Let $\overline{(-)}$ denote the involution of $\ZZ[\irr][t,t^{-1}]$ that fixes $\ZZ[\irr]$ and sends $t$ to $t^{-1}$.
Finally, we define the twisted multiplication \[\grtimt\colon\ZZ[\irr]\twh\times\ZZ[\irr]\twh\ra\ZZ[\irr]\twh\]
via
\[\pi_1\grtimt\pi_2\coloneq t^{\alpha_+(\pi_1,\pi_2)-\langle\pi_2,\pi_1\rangle}\pi_1\grtim\pi_2.\]
Note that by \Cref{L:sym}, $\overline{\pi_1\grtimt\pi_2}=\pi_2\grtimt\pi_1$ since  
\[\alpha_+(\pi_1,\pi_2)+\alpha_+(\pi_2,\pi_1)=(\pi_1,\pi_2)=\langle\pi_1,\pi_2\rangle+\langle\pi_2,\pi_1\rangle.\]
Until associativity is proved, we take iterated products to be
left-associated. For the simultaneous filtration of $\pi_1\times\ldots\times\pi_k$, write
\[\pi_1\times\ldots\times\pi_k\widetilde{[i]}
 \coloneq \pi_1\times\ldots\times\pi_k
 \Bigl[i-\alpha_+(\pi_1,\ldots,\pi_k)+\sum_{a<b}\langle\pi_b,\pi_a\rangle\Bigr].\]
For $k=2$ its semi-simplification gives the coefficient of $t^i$ in $\pi_1\grtimt\pi_2$. In iterated binary products we extend the binary layers additively to semi-simple factors.
The proof of the following theorem is postponed to \Cref{S:JanGeo}.
\begin{theorem}\label{T:assoc2}
    Let $\pi_1,\pi_2,\pi_3\in\irr$ and $k\in\ZZ$.
    Then \[\bigoplus_{j+i=k} \pi_1\times(\pi_2\times\pi_3\widetilde{[j]})\widetilde{[i]} \cong\bigoplus_{j+i=k} (\pi_1\times\pi_2)\widetilde{[j]}\times\pi_3\widetilde{[i]}.\]
\end{theorem}
\begin{rem}
    The two sides use successive binary filtrations, whose independence of the two distinct slopes follows from \Cref{L:ind}. For the simultaneous filtration of a standard module with more than two factors, the direction can matter, see \cite[Appendix A.2]{Ciubotaru2026Jantzen} and \cite[Section 8.3]{Williamson2016LocalHodge}.
\end{rem}
\begin{theorem}\label{T:grprod}
    The map $\grtimt$ is an associative product such that $\pi_1\grtimt\pi_2=\overline{\pi_2\grtimt\pi_1}$.
\end{theorem}
\begin{proof}
    Associativity follows from \Cref{T:assoc2} and the commutativity relation follows from \Cref{L:sym}.
\end{proof}
We now compute this function explicitly in the easiest case, namely of two segments.
\begin{lemma}\label{L:filtseg}
    Let $\De_1$ and $\De_2$ be two segments. Then
    \[\LL(\De_1)\grtim\LL(\De_2)=\begin{cases}
        [\LL(\De_1+\De_2)]&\De_1 \text{ and } \De_2 \text{ are unlinked},\\
        [\LL(\De_1+\De_2)]+[\LL(\De_1\cup\De_2+\De_1\cap\De_2)]t^{-1}&\De_1\prec \De_2,\\
         [\LL(\De_1+\De_2)]+[\LL(\De_1\cup\De_2+\De_1\cap\De_2)]t&\De_2\prec \De_1.
    \end{cases}\]
\end{lemma}
\begin{proof}
    We will use the results of \cite{Dro25}. Namely, by \cite[Lemma 3.8, Corollary 4.3.4]{Dro25} and \Cref{L:filtpol}, the minimal $i$ such that 
    \[\bV_i(\Z(\De_1),\Z(\De_2))=\Z(\De_1)\btimes \Z(\De_2)\]
    is $1$ if $\De_1\prec \De_2$ and $0$ otherwise.
    Therefore, by \Cref{L:sym} the lowest degree in which $\LL(\De_1)\grtim\LL(\De_2)$ does not vanish is $-1$ if $\De_1\prec \De_2$ and $0$ otherwise.
    If the segments are unlinked, the representation $\LL(\De_1)\times\LL(\De_2)\cong \LL(\De_1+\De_2)$ is irreducible by \Cref{T:zel}, which implies that it is concentrated in one degree. By the above, this degree has to be $0$.
    Now by \cite[Theorem 4.2]{Zel80} if $\De_1$ and $\De_2$ are linked, then the induced representation $\Z(\De_1)\times\Z(\De_2)$ is of length $2$ with two irreducible subquotients $\LL(\De_1+\De_2)$ and $\LL(\De_1\cup\De_2+\De_1\cap\De_2)$. Moreover, if $\De_2\prec \De_1$, the unique irreducible subrepresentation is $\LL(\De_1+\De_2)$, and if $\De_1\prec \De_2$, the unique irreducible subrepresentation is $\LL(\De_1\cup\De_2+\De_1\cap\De_2)$. It thus follows that if $\De_2\prec\De_1$, then $\Z(\De_1+\De_2)$ is in degree $0$ and if $\De_1\prec \De_2$, then $\LL(\De_1\cup\De_2+\De_1\cap\De_2)$ is in degree $-1$.
    The remaining degrees follow from the symmetry property of \Cref{T:grprod}.
\end{proof}
 Recall that for $\pi=\Z(\fm),\, \fm=\De_1+\ldots+\De_k$ we defined a standard module \[\I(\fm)= \Z(\De_1)\times\ldots\times \Z(\De_k),\] where for $i>j$, $\De_i\nprec\De_j$, \textit{i.e.} the segments are in an arranged order. 
For an arranged order $\De_1,\ldots,\De_k$, put $\pi_a=\Z(\De_a)$ and set
\[\sigma(\De_1,\ldots,\De_k)\coloneq\sum_{a<b}
\bigl(\alpha_+(\pi_a,\pi_b)-\langle\pi_b,\pi_a\rangle\bigr).\]
We define the graded analogue of the standard module by
\[\Ir(\pi)\coloneq t^{-\sigma(\De_1,\ldots,\De_k)}
  \bigl(\Z(\De_1)\grtimt\ldots\grtimt\Z(\De_k)\bigr).\]
\begin{lemma}\label{L:propI}
    Let $\pi=\Z(\fm)$ and choose an arranged order of $\fm$. Then the following hold.
    \begin{enumerate}
        \item $\Ir(\pi)$ is concentrated in non-positive degrees.
        \item The constant term of $\Ir(\pi)$ equals $[\pi]$.
        \item If $\pi=\Z(\fm)$, then $\pi'=\Z(\fn)$ appears in $\Ir(\fm)$ with a non-zero coefficient if and only if $\fn\le \fm$.
    \end{enumerate}
\end{lemma}
\begin{proof}
First suppose that all segments lie on the same cuspidal line. By \Cref{L:arrangedorderhyp} below,
the coefficient of $\Z(\fn)$ in $\Ir(\fm)$ is the stalk polynomial of $\ic{\oo(\fn)}$ at $T(\fm)$, shifted by $\dim\oo(\fm)$. Here we suppress the common cuspidal label.
On $\oo(\fm)$ the stalk of $\ic{\oo(\fm)}$ is $\CC[\dim\oo(\fm)]$.
If $\oo(\fm)$ is in the boundary of $\oo(\fn)$, the support condition for the intermediate extension gives
\[\HH^j(i_{T(\fm)}^*\ic{\oo(\fn)})=0\quad\text{for } j\ge-\dim\oo(\fm).\]
Outside the closure the stalk vanishes. This proves (1) and (2).
For several cuspidal lines, group the segments by cuspidal line. The block factorization in the proof of \Cref{C:semisimple} identifies the graded product with the product of the one-line calculations, in any interleaving of their orders. The shifts between distinct lines are zero. Thus (1) and (2) follow in this case as well.
For (3), specialize at $t=1$ and apply \cite[Theorem 7.1]{Zel80}. The coefficients before specialization are nonnegative.
\end{proof}
Since we have that $\Z(\fn)$ appears in $\Ir(\fm)$ only if $\fn\le\fm$ and there are only finitely many such $\fn$ for a given $\fm$, the following is immediate.
\begin{corollary}\label{C:gradbas}
    The set $\{\Ir(\pi):\pi\in\irr\}$ is a $\ZZ[t,t^{-1}]$-basis of $\ZZ[\irr][t,t^{-1}]$.
\end{corollary}
\begin{proof}
For each fixed cuspidal support, order the multisegments compatibly with $\le$. By \Cref{L:propI}, the transition matrix from the $\Z(\fm)$ to the $\Ir(\fm)$ is upper triangular with entries in $\ZZ[t^{-1}]$ and ones on the diagonal. Therefore, the inverse is also upper triangular with entries in $\ZZ[t^{-1}]$.
\end{proof}
\subsection{Hecke algebras}\label{S:Hecke}
In this section we will now cast the above results into the language of the underlying Hecke algebras. Note that the description of $\rep_n$ by Hecke algebras is particular to $\gl_n$ and relies on the type-theoretic methods of \cite{BushnellKutzko1998,BushnellKutzko1999}.
Define $H_n(\bq)$ to be the $\CC[\bq,\bq^{-1}]$-algebra generated by
\[S_1,\ldots,S_{n-1}, X_1^{\pm 1},\ldots,X_n^{\pm 1}\]
subject to the following relations.
\begin{enumerate}
    \item $(S_i+1)(S_i-\bq)=0$ $\ain{i}{1}{n-1}$.
    \item $S_iS_j=S_jS_i$ $\ain{i,j}{1}{n-1}$, $\lvert i-j\rvert>1$.
    \item $S_iS_{i+1}S_i=S_{i+1}S_iS_{i+1}$ $\ain{i}{1}{n-2}$.
    \item $X_iX_j=X_jX_i$ $\ain{i,j}{1}{n}$.
    \item $X_jS_i=S_iX_j$ $\ain{i}{1}{n-1}$, $\ain{j}{1}{n}$, $i\notin\{j,j-1\}$.
    \item $S_iX_iS_i=\bq X_{i+1}$ $\ain{i}{1}{n-1}$.
\end{enumerate}
The center of $H_n(\bq)$ is given by $Z_n\coloneq \CC[\bq,\bq^{-1},X_1^{\pm1},\ldots,X_n^{\pm 1}]^{S_n}$.
The specialization to $\bq=q\in\CC^\times$, $q$ not a root of unity, of the Hecke algebra is denoted by \[H_n(q)\coloneq H_n(\bq)\otimes_{\CC[\bq]}\CC[\bq]/_{(\bq-q)}.\]
We recall that every simple $H_n(q)$-module has a central character $\omega$, denoted by an unordered tuple $[\xi_1]+\ldots+[\xi_n],\, \xi_i\in\CC^\times$, and we write $H_n(q)-\Md_{\mathrm{int}}$ for the category of finite-dimensional complex modules of $H_n(q)$ such that every simple subquotient has a central character, denoted $\omega_{q^a}$, represented by an element of the form $[q^{a_1}]+\ldots +[q^{a_n}],\, a\in\ZZ^n$.
As in \Cref{S:class}, we associate to such a simple module $M$ in $H_n(q)-\Md_{\mathrm{int}}$ a vector \[\grdim(M)\in\grdim(Q),\,\grdim(M)_j\coloneq \#\{i:j=a_i\}\]
and write $\omega_\db=\omega_{q^a}$ and \[H_{n,\db}(q)\coloneq H_n(q)\otimes_{Z_n}\omega_{\db}.\]
We write $H_n(q)-\Md_\db$ for the subcategory whose simple subquotients all have central character $\omega_\db$ and $H_n(q)-\irr_\db$ for its simple modules.
For a segment $[a,b]$ of length $n$, let $\Z([a,b])$ denote the character of $H_n(q)$ given by $S_i\mapsto q$ and $X_j\mapsto q^{a+j-1}$.

If $\alpha=(\alpha_1,\ldots,\alpha_k)$ is a partition of $n$, there exists a natural inclusion \[H_\alpha(q)\coloneq H_{\alpha_1}(q)\otimes\ldots\otimes H_{\alpha_k}(q)\hra H_n(q).\] We use the induction functor
\[\id_{H_\alpha(q)}^{H_n(q)}\colon H_\alpha(q)-\Md\ra H_n(q)-\Md, M\mapsto\hom_{H_\alpha(q)}(H_n(q),M).\]
We also use the notation $\times$ analogously to the $p$-adic setting and write $H_\alpha(q)-\Md_{\mathrm{int}}$ for the subcategory with integral central characters in each block.
For $\rho\in\cuss$ with $d=\deg(\rho)$, we denote by $\rep_{n,\rho}$ the full subcategory of $\rep_{nd}$ consisting of objects whose irreducible subquotients are of the form $\Z(\fm_\rho)$, $\fm\in\mul$. Note that parabolic induction induces for a partition $\alpha=(\alpha_1,\ldots,\alpha_k)$ of $n$ a functor
\[\times\colon \rep_{\alpha,\rho}\coloneq \rep_{\alpha_1,\rho}\times \ldots\times\rep_{\alpha_k,\rho}\ra \rep_{n,\rho}.\]
\begin{theorem}[{\cite{BushnellKutzko1998},\cite{BushnellKutzko1999}}]
Let $\rho\in\cuss$. Then there exists $q(\rho)=q^{f(\rho)}$, a power of $q$, and an equivalence of abelian categories
    \[\ct_{n,\rho}\colon \rep_{n,\rho}\ra H_n(q(\rho))-\Md_{\mathrm{int}}\]
    such that the following diagram commutes.
    \[\begin{tikzcd}
       \rep_{\alpha,\rho}\arrow[d,"\times"]\arrow[r,"\ct_{\alpha,\rho}"]&  H_\alpha(q(\rho))-\Md_{\mathrm{int}}\arrow[d,"\times"]\\\rep_{n,\rho}\arrow[r,"\ct_{n,\rho}"]&H_n(q(\rho))-\Md_{\mathrm{int}}
    \end{tikzcd}\]
    Moreover, $\ct_{l(\De),\rho}(\Z(\De_\rho))=\Z(\De)$ and it induces for all $\db\in\grdim(Q)$ a bijection
    \[\ct_{n,\rho}\colon \irr_{\db,\rho}\ra H_n(q(\rho))-\irr_{\db}.\]
\end{theorem}
We parametrize the simple modules of $H_n(q)$ with central character $\omega_\db$ by the bijection
\[\Z\colon \mul_\db\ra H_n(q)-\irr_\db\]
for which
\[\ct_{n,\rho}(\Z(\fm_\rho))=\Z(\fm).\]
 Let $M=M_1\otimes\ldots\otimes M_k\in H_\alpha(q)-\Md_{\mathrm{int}}$ and let $\lambda\in \ZZ^k$ be generic, as in \Cref{S:int}.
Then we let 
\[M_1\times\ldots\times M_k[[s]]_\lambda\]
denote the $\CC[[s]]$-$H_n(q)$-module induced from the $H_\alpha(q)$-module
\[(M_1\otimes \CC[[s]])\otimes\ldots\otimes (M_k\otimes \CC[[s]]),\]
where the action of $H_{\alpha_i}(q)$ on $M_i\otimes \CC[[s]]$ is given by multiplying the action of each Bernstein generator $X_j$ in the $i$th block by
\[\exp(\lambda_i s)\coloneq \sum_{k\ge 0}\frac{(\lambda_is)^k}{k!}.\]
Note that for $\pi_1\otimes\ldots\otimes\pi_k\in\rep_{\alpha,\rho}$ the intertwining operator
\[J_{\pi_1,\ldots,\pi_k}(s)\circ\mathrm{Ad}(w_{\overline\alpha})\colon \pi_k\otimes \lvert-\rvert^{\lambda_ks}\times\ldots\times \pi_1\otimes \lvert-\rvert^{\lambda_1s}\ra \pi_1\otimes \lvert-\rvert^{\lambda_1s}\times\ldots\times \pi_k\otimes \lvert-\rvert^{\lambda_ks}\]
is well defined for generic $s\in\CC$ and for a formal variable $s$, see \cite[§7]{Dat05}, and hence gives rise to a $\CC((s))$-linear morphism
\[\pi_k\otimes_{\CC((s))}\lvert-\rvert^{\lambda_ks}\times\ldots\times\pi_1\otimes_{\CC((s))}\lvert-\rvert^{\lambda_1s}\ra \pi_1\otimes_{\CC((s))}\lvert-\rvert^{\lambda_1s}\times\ldots\times\pi_k\otimes_{\CC((s))}\lvert-\rvert^{\lambda_ks},\]
where the twist multiplies each $X_j$ in block $i$ by $q(\rho)^{\lambda_i s}$. For $M_1\otimes\ldots\otimes M_k=\ct_{\alpha,\rho}(\pi_1\otimes\ldots\otimes\pi_k)$, denote its image under $\ct_{n,\rho}$ at $s/\log q$ by
\[J_{M_1,\ldots,M_k}(s)\colon M_k\times\ldots\times M_1((s))_{f(\rho)\cdot\overline{\lambda}}\ra M_1\times\ldots\times M_k((s))_{f(\rho)\cdot\lambda}.\]
We can define analogous filtrations as in \Cref{S:Jantzen} denoted by
$V_i(M_1,\ldots,M_k)$ and note that by construction
\[\ct_{n,\rho}(V_i(\pi_1,\ldots,\pi_k))=V_i(M_1,\ldots,M_k).\]
If we are in a situation as above we can also consider the $\CC[[s]]$-linear character $\omega_{\db,s,\lambda}$ of $Z_n$ obtained by twisting $\omega_{\db}$ on the first $\alpha_1$ entries by $\exp(\lambda_1s)$, the second $\alpha_2$ entries of $\omega_\db$ by $\exp(\lambda_2s)$, and so on. We then set
\[H_{n,\db,s,\lambda}(q)=H_n(q)\otimes_{Z_n}\omega_{\db,s,\lambda}.\]
Note that $H_{n,\db,s,\lambda}(q)\lvert_{s=0}=H_{n,\db}(q)$.
\section{Generalized Rogawski conjecture}\label{S:JanGeo}
In this section we will show how the product $\grtimt$ on $\ZZ[\irr][t,t^{-1}]$ can be described geometrically in terms of hyperbolic localization functors on $\ep$.

Let $\db=\db^1+\ldots+\db^k$, $\db^i\in\grdim(Q)$ and $\alpha=(\lvert\db^1\rvert,\ldots,\lvert\db^k\rvert)$.
We reinterpret the $\gdb$-variety $\ep$ as follows. Let $\cn$ be the complex variety of nilpotent matrices in $\mathfrak{gl}_n$, $n=\lvert\db\rvert$, on which $\gl_n\times\gm$ acts by conjugation and scaling. To $\db$ we associate a cocharacter \[\chi_\db\colon \gm\ra \gl_n\times\gm, \,t\mapsto \mathrm{diag}(\chi_a,t^{-1}), \chi_a=(t^{a_1},\ldots,t^{a_n}),\,\]\[a=(\ldots,\overbrace{i,\ldots,i}^{\db_{i}^1},\overbrace{i+1,\ldots,i+1}^{\db_{i+1}^1},\ldots, \overbrace{i,\ldots,i}^{\db_{i}^2},\ldots,\overbrace{i,\ldots,i}^{\db_{i}^k},\ldots)\in\ZZ^n.\]
Let $V_i$ be the subspace of $\CC^n$ given by \[\{v\in\CC^n:\chi_a(t)\cdot v=t^iv\}.\]
The graded vector space $V\coloneq \bigoplus_{i\in\ZZ}V_i$ has dimension vector $\db$ and
for any $x\in\cn$ fixed by $\chi_\db$, we have $x(V_i)\subseteq V_{i+1}$. This allows us to identify the $\chi_\db$-fixed points in $\cn$ with $\ep$. Moreover, the centralizer of $\chi_\db$ in $\gl_n$ acts on the $\chi_\db$-fixed points in $\cn$, and this centralizer is, under the above identification, isomorphic to $\gdb$.

For a generic weight $\lambda\in\ZZ^k$ consider the cocharacter \[\chi_\lambda\colon \gm\ra \gl_n, t\mapsto (\overbrace{t^{\lambda_1},\ldots,t^{\lambda_1}}^{\alpha_1},\ldots,\overbrace{t^{\lambda_k},\ldots,t^{\lambda_k}}^{\alpha_k}).\]
We let $\ep^0$ be the fixed locus of $\chi_\lambda$ and \[\ep^\pm=\{x\in \ep:\lim_{t\ra 0}\chi_\lambda(t^{\pm 1})x\text{ exists}\}\] be its attracting and repelling loci, which have the structure of complex varieties.
There exist natural maps $p^\pm\colon \ep^\pm\ra\ep^0,\,x\mapsto \lim_{t\ra 0}\chi_\lambda(t^{\pm 1})x$.
\[\begin{tikzcd}
    \ep^0&\ep^+\arrow[l,"p^+"]\arrow[d,"i_+",hookrightarrow]\\ \ep^-\arrow[u,"p^-"]\arrow[r,"i_-",hookrightarrow]&\ep
\end{tikzcd}\]
Finally, we let $\gdb^\pm$ be the attracting and repelling loci for the action of $\chi_\lambda$ on $\gdb$, respectively, and $\gdb^0$ its fixed locus. Recall that the $\gdb^\pm$ are opposite parabolic subgroups and $\gdb^0$ is their Levi component. 
\begin{lemma}
    The $\gdb^0$-variety $\ep^0$ is naturally isomorphic to the $\bG_{\db^1}\times\ldots\times \bG_{\db^k}$-variety $\epo\times\ldots\times\epk$.
\end{lemma}
\begin{proof}
    Recall the identification of $\ep$ with the $\chi_\db$-fixed points on $\cn$. One can reverse the order of taking fixed points, \textit{i.e.}, first consider the $\chi_\lambda$- and then the $\chi_\db$-fixed points on $\cn$. It is clear that for generic $\lambda$, the fixed points are isomorphic to the $\bM_\alpha\times\gm$-variety $\mathbf{N}_{\alpha_1}\times\ldots\times\mathbf{N}_{\alpha_k}$. Taking the $\chi_\db$-fixed points gives the desired identification.
\end{proof}
\subsection{Hyperbolic localization}\label{S:hyploc}
We will assume from now on that $\lambda$ satisfies $\lambda_1>\ldots>\lambda_k$.
We recall the hyperbolic localization functor of \cite{braden2003hyperbolic}, see also \cite{drinfeld2014theorem}.
We denote by $\bU_\db^\pm$ the unipotent part of $\gdb^\pm$.
\[\hyp^+\coloneq \mathrm{Inv}_{\bU_\db^+,!}\circ p_!^+\circ i_+^*\circ\mathrm{For}_{\gdb}^{\gdb^+}\colon D_c^b(\ep,\gdb)\ra D_c^b(\ep^0,\gdb^0)\]
\[\hyp ^-\coloneq \mathrm{Inv}_{\bU_\db^-,!}\circ p_!^-\circ i_-^*\circ \mathrm{For}_{\gdb}^{\gdb^-}\colon D_c^b(\ep,\gdb)\ra D_c^b(\ep^0,\gdb^0).\]
Let $S_k$ act on $\ZZ^k$ by permuting the entries. Then $\hyp^+$ and $\hyp^-$ depend only on the $S_k$-chamber in which $\lambda$ lies. If at some point we want to highlight the choice of $\lambda$, we write $\hyp_\lambda^\pm$.
Note that it is not clear \emph{a priori} that the images of these functors are contained in the bounded derived categories. That they are well defined is a consequence of the following theorem. 
\begin{theorem}[{Braden's theorem, \cite[Theorem 1]{braden2003hyperbolic}, \cite{drinfeld2014theorem}}]\label{T:BraThm}
    There are natural isomorphisms
    \[\hyp^+\cong \mathrm{Inv}_{\bU_\db^-,*}\circ p_*^-\circ i_-^!\circ\mathrm{For}_{\bG_\db}^{\bG_\db^-},\]
    \[\hyp^-\cong \mathrm{Inv}_{\bU_\db^+,*}\circ p_*^+\circ i_+^!\circ\mathrm{For}_{\bG_\db}^{\bG_\db^+}.\]
    Moreover,
    \[\DD\circ\hyp^+\cong \hyp^-\circ\DD\]
    and the functors restrict to
    \[\hyp^\pm\colon \sem(\ep,\gdb)\ra \sem(\ep^0,\gdb^0).\]
\end{theorem}
\begin{proof}
Restrict equivariance to $\gdb^0$ and apply Braden's theorem on $\ep$ for $\chi_\lambda$. Since $\bU_\db^\pm$ is unipotent and acts trivially on $\ep^0$, inflation from $\gdb^0$ to $\gdb^\pm$ is an equivalence on $\ep^0$. Both its left and its right adjoint are inverse to it, so the resulting comparison is the one displayed above.
\end{proof}

We recall that the reason the second part is true is that any $\ic{\oo}$ on the source can be lifted to a pure Hodge module, and the first part implies its image is again a pure mixed Hodge module. The claim then follows from \Cref{T:decomp}.
The following two lemmas follow immediately from the definitions.
\begin{lemma}\label{L:hypassos}
Assume $k=3$.
    The following diagram commutes.
\[\begin{tikzcd}
    \sem(\ep,\gdb)\arrow[r,"\hyp^+"]\arrow[d,"\hyp^+"]\arrow[dr,"\hyp^+"]&\sem(\mathbf{E}_{\db^1}\times \mathbf{E}_{\db^2+\db^3}, \mathbf{G}_{\db^1}\times \mathbf{G}_{\db^2+\db^3})\arrow[d,"\mathbbm{1}\boxtimes \hyp^+"]\\
    \sem( \mathbf{E}_{\db^1+\db^2}\times\mathbf{E}_{\db^3}, \mathbf{G}_{\db^1+\db^2}\times \mathbf{G}_{\db^3})\arrow[r,"\hyp^+\boxtimes \mathbbm{1}"]&  \sem( \mathbf{E}_{\db^1}\times\mathbf{E}_{\db^2}\times\mathbf{E}_{\db^3}, \mathbf{G}_{\db^1}\times\mathbf{G}_{\db^2}\times \mathbf{G}_{\db^3}) 
\end{tikzcd}\]
\end{lemma}
\begin{lemma}\label{L:hypsym}
    For $k=2$, let $w\in\gdb$ be the block permutation at each vertex that reverses the two ordered blocks and preserves the order within each block. Then
    the following diagram commutes.
    \[
    \begin{tikzcd}
        \sem(\ep,\gdb)\arrow[rr,"\hyp^+"]\arrow[d,"(\cdot w)^*"]&&\sem(\mathbf{E}_{\db^1}\times\mathbf{E}_{\db^2},\bG_{\db^1}\times \bG_{\db^2})\arrow[d,"\cong"]\\
         \sem(\ep,\gdb)\arrow[rr,"\hyp^-"]&&\sem(\mathbf{E}_{\db^2}\times\mathbf{E}_{\db^1},\bG_{\db^2}\times \bG_{\db^1})
    \end{tikzcd}
    \]
\end{lemma}
Since the sheaves on the left side are $\gdb$-equivariant, they remain unchanged under the pullback of the map given by the action of $w$.
\subsection{Geometric realization of Hecke algebras}
In this section we recall some results of \cite{ChrissGinzburg1997} on how to describe the affine Hecke algebra $H_n(\bq)$ in terms of the Borel--Moore homology of the Steinberg variety and the Yoneda algebra of the Springer sheaf, and how simple modules are parametrized by simple $\gdb$-equivariant perverse sheaves on $\ep$. We also recall from \cite{Ciubotaru2026Jantzen} and \cite{FujitaHernandez2026} how the Jantzen filtration is realized geometrically in the case of standard modules.

We start by recalling the Springer resolution
\[\mu\colon \tep\ra \ep.\]
It is constructed as follows.
Let $\mu\colon \Tilde{\mathbf{N}}_n\ra \cn$ be the classical Springer resolution, \textit{i.e.}, $\Tilde{\mathbf{N}}_n$ consists of pairs $\{(x,\mathfrak{b}):x\in\cn\cap\mathfrak{b}\}$, where $\mathfrak{b}$ ranges over the Lie algebras of Borel subgroups of $\gl_n$. Note that $\mu^{-1}(0)\cong \cb\coloneq \gl_n/\bB_n$.
The action of $\gl_n\times\gm$ lifts to an action on $\Tilde{\mathbf{N}}_n$, and taking the $\chi_\db$-fixed points defines \[\mu\colon \tep\ra \ep.\]
Let $\fm\in\mul$ and write $\fm=\De_1+\ldots+\De_k$ in an arranged order. We let $\db^i$ be the dimension vector of $\De_i$ and $\db=\db^1+\ldots+\db^k$ and recall the cocharacter $\chi_\lambda$ of $\gl_n$ defined in \Cref{S:JanGeo}. Throughout this section we assume that $\lambda_1>\ldots>\lambda_k$.
We denote the fixed locus of $\chi_\lambda$ on $\tep$ by $\tep^0$ and its attracting and repelling loci by $\tep^\pm$.
The following diagram then commutes.
\[\begin{tikzcd}
\tep^0\arrow[d,"\mu^0"]&\tep^\pm\arrow[r,"i_\pm",hookrightarrow]\arrow[l,"p^\pm"]\arrow[d,"\mu^\pm"]&\tep\arrow[d,"\mu"]\\
    \ep^0&\ep^\pm\arrow[r,"i_\pm",hookrightarrow]\arrow[l,"p^\pm"]&\ep
\end{tikzcd}\]
If we work with arbitrary regular weights $\lambda$, we decorate the respective maps with an index $\lambda$.
\subsubsection{Sheaf-theoretic interpretation}
We recall the Springer sheaf $\cS_\db\coloneq \mu_*\bigl(\bigoplus_{\boldsymbol i}\underline{\CC}_{\tep^{\boldsymbol i}}[\dim\tep^{\boldsymbol i}]\bigr)$, where $\tep^{\boldsymbol i}$ is the component indexed by the flag word $\boldsymbol i$.
By the decomposition theorem of \cite{BBD1982}, \[\cS_\db=\bigoplus_{k\in\ZZ,\fm\in\mul_\db}L_{\oo(\fm),k}\otimes\ic{\oo(\fm)}[k].\]
We set $L_{\fm}\coloneq\bigoplus_{k\in\ZZ} L_{\oo(\fm),k}$ and denote by \[\ext^\bullet(\cS_\db,\cS_\db)\] the Yoneda algebra of $\cS_\db$ in $D_c^b(\ep)$, which is isomorphic as an ungraded algebra to
\[\bigoplus_{\fm,\fn\in\mul_\db,k\in\ZZ}\mathrm{Hom}(L_{\fm},L_{\fn})\otimes \ext^k(\ic{\oo(\fm)},\ic{\oo(\fn)}),\]
see \cite[p. 447]{ChrissGinzburg1997},
and hence admits the algebra \[\bigoplus_{\fm\in\mul_\db}\ed(L_\fm)\] as a quotient. Thus the set $\{L_\fm\}_{\fm\in\mul_\db}$ is a family of simple $\ext^*(\cS_\db,\cS_\db)$-modules, see \cite[Theorem 8.6.12]{ChrissGinzburg1997}, and $\Z(\fm)\cong L_\fm$.

We again let $\gm$ act via $\chi_\lambda$ on $\ep$ and consider the direct product of the $\gm$-equivariant $\ext$-algebra 
\[\widehat{\ext}_{\gm}^\bullet(\cS_\db,\cS_\db).\]
The following is the equivariant version of \cite[Theorem 8.6.7]{ChrissGinzburg1997} and admits exactly the same proof.
\begin{theorem}
    There exists an isomorphism of $\CC[[s]]$-algebras
    \[\real_{const,s}\colon H_{n,\db,s,\lambda}(q)\ra \widehat{\ext}_{\gm}^\bullet(\cS_\db,\cS_\db).\]
\end{theorem}
Recall that we identified $\ep$ with the moduli space of the quiver $Q=A_\infty$ with dimension vector $\db$.

Let $\fm\in\mul_\db$ and write $\fm=\De_1+\ldots+\De_k$.
We let $\db^j$ be the dimension vector of $\De_j$. The point
$T(\fm)=T(\De_1)\oplus\ldots\oplus T(\De_k)$ is fixed by $\chi_\lambda$,
and we denote its inclusion into $\ep$ by $i_{T(\fm)}$.
We then have two natural $\widehat{\ext}_{\gm}^\bullet(\cS_\db,\cS_\db)$-modules
\[\widehat{\HH}_{\gm}^\bullet(i_{T(\fm)}^!\cS_\db),\,
  \widehat{\HH}_{\gm}^\bullet(i_{T(\fm)}^*\cS_\db).\]
The comparison morphism for the filtration is the costalk-to-stalk map on a transverse slice at $T(\fm)$.
We now come to the Rogawski conjecture of \cite{Rogawski1985}, which links the simultaneous Jantzen filtration of the standard module $\I(\fm)$ to the dimensions of stalks of simple perverse sheaves on $\ep$. For degenerate affine Hecke algebras of type $A$ this was proven in \cite{Suzuki1998} and it was announced without details for arbitrary Hecke algebras in \cite{Ginzburg1987}. Recently, proofs of it were given in \cite{FujitaHernandez2026} and \cite{Ciubotaru2026Jantzen}.
\begin{theorem}[Rogawski conjecture]\label{T:RoCon}
Order the segments $\De_1,\ldots,\De_k$ by increasing right endpoint
and, for equal right endpoints, by increasing left endpoint. This is an
adapted order for $Q$ in the coinduction convention of \Cref{S:Hecke}
\cite[Lemma~7.3]{FujitaHernandez2026}.
For all $\lambda\in\ZZ^k$ there exist isomorphisms of $H_{n,\db,s,\lambda}(q)$-modules
\[\widehat{\HH}_{\gm}^\bullet(i_{T(\fm)}^!\cS_\db)\cong L_{\De_1}\times\ldots\times L_{\De_k}[[s]]_\lambda,\]
\[\widehat{\HH}_{\gm}^\bullet(i_{T(\fm)}^*\cS_\db)\cong L_{\De_k}\times\ldots\times L_{\De_1}[[s]]_{\overline{\lambda}}.\]
Furthermore, if $\lambda_1>\ldots>\lambda_k$, or more generally if $\lambda$ lies in a suitable Zariski-open dense subset of $\CC^k$ and has integral entries, $L_{\De_1}\times\ldots\times L_{\De_k}[j]$ is semi-simple for all $j\in\ZZ$ and the multiplicity of $L_{\fn}$ in $L_{\De_1}\times\ldots\times L_{\De_k}[j]$ equals
\[\dim_\CC\mathbb{H}^{j-\dim\oo(\fm)}(i_{T(\fm)}^*\ic{\oo(\fn)}).\]
\end{theorem}
\begin{proof}
For the first two isomorphisms, restrict the block-torus action along
$z_j\mapsto\lambda_j s$ in \cite[proof of Proposition~7.13 and Corollary~7.16]{FujitaHernandez2026}.
See also \cite[Lemma~4.11]{Ciubotaru2026Jantzen}.
For the filtration, first take $\lambda_1>\ldots>\lambda_k$. The transverse slice at $T(\fm)$ in \cite[\S7.4]{FujitaHernandez2026} has normal space $\bigoplus_{a,b}\ext_Q^1(T(\De_a),T(\De_b))$. A nonzero summand has $a<b$ and weight $\lambda_b-\lambda_a$. Hence the inverse cocharacter contracts the slice. Changing $s$ to $-s$ preserves the filtration. In the proof of \cite[Propositions~7.13 and 7.17]{FujitaHernandez2026}, one may therefore restrict the torus on the individual segment factors along these slopes. Each elementary exchange has parameter $(\lambda_a-\lambda_b)s$, which changes its primitive normalization by a unit. For equal segments the primitive exchange is an isomorphism, and distinct slopes allow the original meromorphic exchange to be restricted. By \Cref{L:filtseg,L:filtpol}, for $a<b$ the primitive exchange in this order is obtained by multiplying the analytic exchange by $s^{\alpha_+(\Z(\De_a),\Z(\De_b))}$. Multiplying along the block reversal gives the normalization of \Cref{S:Jantzen}. For the standard order the filtration of \cite[Example~2.15]{FujitaHernandez2026} is the filtration of this single map. The slice calculation and semisimplicity now follow as in \cite[Theorem~5.6, Theorem~7.18 and Corollary~7.19]{FujitaHernandez2026}. The shift by $\dim\oo(\fm)$ converts the stalk on the ambient variety to the stalk on the slice.

For completeness, the assertion for a Zariski-open dense set follows here from the result on this chamber. The standard module at $s=0$ is fixed. In the finite-dimensional Hecke model the intertwiner is rational in the block twists $e^{\lambda_a s}$, so its Laurent coefficients are rational in $\lambda$ wherever the leading denominators do not vanish. On a dense open set the pole orders of the matrix and its inverse are bounded. All layers are then determined by kernels and images of finite Taylor-jet matrices. After shrinking this set to make their ranks constant, the layers and the actions of the Hecke generators form algebraic families. For each of the finitely many layers, being isomorphic to the specified semi-simple module is a constructible condition, since a change-of-basis orbit is locally closed. These conditions hold simultaneously on the integral points of the strictly ordered chamber in this open set. Those points are Zariski dense, so their common constructible locus contains a dense open subset.
\end{proof}
Now fix dimension vectors $\db^1,\ldots,\db^k\in\grdim(Q)$ and $\rho\in \cuss$. Let $\db=\db^1+\ldots+\db^k$ and $\lambda\in\ZZ^k$ as always.
We consider $\sem(\ep,\gdb)$ as a $\ZZ\twh$-module via $t\cdot \cf=\cf[-1]$.
Let ${\ZZ[\irr_{\db,\rho}]\twh}^*$ be the dual $\ZZ\twh$-module of $\ZZ[\irr_{\db,\rho}]\twh$.
Define an isomorphism of $\ZZ\twh$-modules
\[\ct_{\db,\rho}\colon \sem(\ep,\gdb)\ra {\ZZ[\irr_{\db,\rho}]\twh}^*\]
sending for $\fm\in \mul_\db$ the sheaf $\ic{\oo(\fm)}$
to the function
\[\delta_{\fm_\rho}(\Z(\fn_\rho))=\begin{cases}
    1&\text{ if } \fn=\fm,\\0&\text{ otherwise}.
\end{cases}\]
We define similarly
\[\ct_{\db^1,\ldots,\db^k,\rho}\colon \sem(\epo\times\ldots\times\epk,\bG_{\db^1}\times\ldots\times \bG_{\db^k})\ra {\ZZ[\irr_{\db^1,\rho}\times\ldots\times \irr_{\db^k,\rho}]\twh}^*.\]
\begin{lemma}\label{L:arrangedorderhyp}
    Let $\fm,\fn\in\mul_\db$, $\fm=\De_1+\ldots+\De_k$ in arranged form and $\db^j$ the dimension vector of $\De_j$.
    Then
    \begin{multline*}
       \ct_{\db,\rho}(\ic{\oo(\fn)})(L_{\De_1}\grtimt\ldots\grtimt L_{\De_k})\\
       =\ct_{\db^1,\ldots,\db^k,\rho}(\hyp^+(\ic{\oo(\fn)}))(L_{\De_1}\otimes\ldots\otimes L_{\De_k}).
    \end{multline*}
In particular, with the normalization of \Cref{S:Jantzen},
\[\ct_{\db,\rho}(\ic{\oo(\fn)})(\Ir(\fm_\rho))
=\sum_{j\in\ZZ}\dim_\CC\HH^j(i_{T(\fm)}^*\ic{\oo(\fn)})\,t^{j+\dim\oo(\fm)}.\]
\end{lemma}
\begin{proof}
The first identity follows by evaluating the diagram of \Cref{T:genrog} below on $L_{\De_1}\otimes\ldots\otimes L_{\De_k}$. Its proof is independent of this lemma. We compute the right-hand side from the stalk at $T(\fm)$.

Write $h_{ab}=\dim_\CC\hom_Q(T(\De_a),T(\De_b))$. By \Cref{S:class},
\[\sigma(\De_1,\ldots,\De_k)=\sum_{a<b}(h_{ab}-h_{ba}).\]
Since the segments are in an arranged order, $\ext_Q^1(T(\De_b),T(\De_a))=0$ for $b\ge a$.
An element of $(p^+)^{-1}(T(\fm))$ gives a filtered representation
\[0=F_0\subseteq F_1\subseteq\ldots\subseteq F_k=x,\quad F_a/F_{a-1}=T(\De_a),\]
with the indicated identifications on the graded pieces. The Ext vanishing implies successively that these extensions split. Hence the fibre is the $\bU_\db^+$-orbit of $T(\fm)$. Its stabilizer has dimension $\sum_{a<b}h_{ba}$, and the fibre is an affine space of dimension
\[f=\dim\bU_\db^+-\sum_{a<b}h_{ba}.\]

On the other hand, note that $\mathbf E_{\db^a}$ has a unique open orbit, $\oo(\De_a)$. Put $D=\sum_a\dim\oo(\De_a)$. It thus follows that the multiplicity of $(\ic{\oo(\De_1)}\boxtimes\ldots\boxtimes\ic{\oo(\De_k)})[-i]$
in $\hyp^+\ic{\oo(\fn)}$ is the dimension of the stalk of $\hyp^+\ic{\oo(\fn)}$ at $T(\fm)$ in degree $i-D$.
By equivariance the restriction of $\ic{\oo(\fn)}$ to the fibre is constant.
Compactly supported base change therefore gives this multiplicity as
\[\dim_\CC\HH^{i-D-2f}(i_{T(\fm)}^*\ic{\oo(\fn)}).\]
Finally, $\dim\oo(\fm)=\dim\gdb-\sum_{a,b}h_{ab}$ and $\dim\gdb-\dim\gdb^0=2\dim\bU_\db^+$, so
\[D+2f=\dim\oo(\fm)+\sum_{a<b}(h_{ab}-h_{ba}).\]
Multiplying the tilde product by $t^{-\sigma(\De_1,\ldots,\De_k)}$ gives the stated stalk formula.
\end{proof}
By \Cref{L:hypassos}, the graded multiplicities in this calculation are independent of the parenthesization of the tilde product.

\begin{theorem}\label{T:genrog}
        The following diagram commutes.
    \[\begin{tikzcd}
   \sem(\ep,\gdb)\arrow[rrr,"\hyp^+"]\arrow[d,"\ct_{\db,\rho}"]&&&\sem(\mathbf{E}_{\db^1}\times\ldots\times \mathbf{E}_{\db^k},\bG_{\db^1}\times\ldots\times \bG_{\db^k})\arrow[d,"\ct_{\db^1,\ldots,\db^k,\rho}"]\\
        {\ZZ[\irr_{\db,\rho}][t,t^{-1}]}^*\arrow[rrr,"{\grtimt}^*"]&&&{\ZZ[\irr_{\db^1,\rho}\times\ldots\times \irr_{\db^k,\rho}][t,t^{-1}]}^*.
    \end{tikzcd}\]
\end{theorem}
\begin{proof}
First let $k=2$ and $\pi_j=\Z_\rho(\oo_j)$ for $j=1,2$. By \Cref{L:ind}
we may take $\lambda=(1,0)$. Use the grading on $\cS_\db$ of Appendix~\ref{app:intertwiner}.
On the completed section lattices, \Cref{prop:intertwiner-hyp} identifies
$J_{\pi_1,\pi_2,1,0}(s)$, up to a unit, with
$s^{-\langle\db^2,\db^1\rangle}R_{\mathrm{hyp}}^{-1}$ on the
$\ic{\oo_1\times\oo_2}$-multiplicity spaces of $\hyp^\pm\cS_\db$.
By \Cref{L:smith-hyp}, each copy of $\ic{\oo_1\times\oo_2}[-n]$ in
$\hyp^+\ic{\oo}$ gives an elementary divisor $s^n$ of $R_{\mathrm{hyp}}$,
and these account for all elementary divisors. The corresponding divisor
of the analytic intertwiner is $s^{-n-\langle\db^2,\db^1\rangle}$ up to
a unit.
To read off the multiplicity of $L_{\fm}$, take the idempotent projecting $\cS_\db$ onto a single shifted copy of $P_\oo=\ic{\oo}$ for $\oo=\oo(\fm)$. It has rank one on $L_\fm$ and acts by zero on the other simple modules. Applying this idempotent to a layer is exact, and by naturality its image is the corresponding layer on the coefficient of that copy of $P_\oo$. In twisted degree $n$, its dimension is therefore the multiplicity of the divisor $s^n$ for $R_{\mathrm{hyp}}(P_\oo)$. This proves the diagram for two factors.

For $k>2$, first separate the last block from the first $k-1$ blocks.
Applying the two-factor result to each shifted simple summand and
iterating computes the left-associated product. By \Cref{L:hypassos}, either parenthesization gives the same graded multiplicities.
\end{proof}

We now define for $\pi_1,\ldots,\pi_k\in\irr$ and $i\in\ZZ$ recursively
\[\pi_1\times\pi_2[i]'\coloneq \pi_1\times\pi_2\widetilde{[i]},\,\pi_1\times\ldots\times \pi_k[i]'\coloneq \bigoplus_{r+s=i}(\pi_1\times\ldots\times \pi_{k-1}[r]')\times \pi_k\widetilde{[s]}.\]
Note that in general it is not clear whether the recursively defined layer is isomorphic to the corresponding layer of the simultaneous filtration, that is, whether
\[\pi_1\times\ldots\times \pi_k[i]'\cong\pi_1\times\ldots\times \pi_k\widetilde{[i]}.\]
\begin{corollary}\label{C:semisimple}
    Let $\pi_1,\ldots,\pi_k\in\irr$. Then
    for all $i\in\ZZ$ we have that $\pi_1\times\ldots\times \pi_k[i]'$ is semi-simple. In particular, the image of
    \[J_{\pi_1,\pi_2}\colon\pi_1\btimes\pi_2\ra\pi_1\times\pi_2\]
    is semi-simple.
\end{corollary}
\begin{proof}
By additivity and the recursive definition, it suffices to treat two
irreducible factors. Suppose first that they lie on one cuspidal line. We use an argument similarly to the proof of
\cite[Theorem 8.6.23]{ChrissGinzburg1997}. Namely, the cohomological degrees give the induced representation a filtration, which agrees with the Jantzen filtration, and the Ext-algebra of the Springer sheaf acts on the graded module via its projection to $\bigoplus_{\oo\in \orb_\db}\mathrm{End}(L_\oo)$. To be more precise and thanks \Cref{prop:intertwiner-hyp} and \Cref{L:smith-hyp},
the contribution of \(\ic{\oo}[a]\) to the layer of tilde degree \(i\)
is concentrated in cohomological degree \(-i-a\). Thus a matrix entry
\(\ic{\oo}[a]\to \ic{\oo}[b+p]\) of degree \(p\) can act nontrivially
only if \(-i-a+p=-i-b\). Hence its Ext degree \(p+b-a\) between
the simple perverse sheaves is zero. The action therefore factors
through \(\bigoplus_{\fm}\ed(L_{\fm})\), so the layer is semi-simple.

For arbitrary cuspidal support, the intertwiner factors over cuspidal
lines \cite[Proposition~7.8]{Dat05} and hence the claim follows easily.
\end{proof}

\begin{proof}[Proof of \Cref{T:assoc2}]
If the three factors lie on a single cuspidal line, \Cref{T:genrog,L:hypassos} give the same multiplicities in each degree for both parenthesizations. For arbitrary cuspidal support, the same equality follows by factoring over cuspidal lines as in the proof of \Cref{C:semisimple}. By the two-factor case of that corollary and additivity, the layers for both parenthesizations are semi-simple, hence isomorphic in each degree.
\end{proof}

We now recall the operation, see for example \cite[§10.3]{Achar2021Perverse},
\[\conv\colon \sem(\mathbf{E}_{\db^1}\times\ldots\times \mathbf{E}_{\db^k},\bG_{\db^1}\times\ldots\times\bG_{\db^k}
)\ra\sem(\ep,\bG_\db)\]
defined as follows.
Consider the diagrams
\[\begin{tikzcd}
    \epo\times\ldots\times \epk&\arrow[l,"p^+"]\ep^+\arrow[r,"i_+"]&\ep\\ \gdb\overset{\gdb^+}{\times}\ep\arrow[r,"m"]&\ep
\end{tikzcd}\]
and define the functor
\[T\coloneq {i_+}_*\circ (p^+)^!\circ {\mathrm{Inf}}_{\bG_{\db}^0}^{\gdb^+}\] between \[D_c^b(\mathbf{E}_{\db^1}\times\ldots\times\mathbf{E}_{\db^k},\bG_{\db^1}\times\ldots\times\bG_{\db^k})\ra D_c^b(\ep,\gdb^+).\]
Recall the isomorphism \[S\colon D_c^b(\ep,\gdb^+)\ra D_c^b(\gdb\overset{\gdb^+}{\times}\ep,\gdb).\]
We then set
\[\conv\coloneq m_*\circ S\circ T\colon D_c^b(\mathbf{E}_{\db^1}\times\ldots\times\mathbf{E}_{\db^k},\bG_{\db^1}\times\ldots\times\bG_{\db^k})\ra D_c^b(\ep,\gdb).\]
It acts as the geometrization of the multiplication of the quantum algebra associated to the quiver $Q$, see for example \cite{Lusztig1991Quivers}.
\begin{lemma}[{\cite[§10.5]{Achar2021Perverse}, \cite[Theorem 1]{braden2003hyperbolic}}]\label{L:adjunct}
    The convolution $\conv$ is well defined and the right adjoint to $\hyp^+$.
\end{lemma}
\begin{proof}
    The fact that $\conv$ takes simple perverse sheaves to semi-simple complexes of shifted simple perverse sheaves follows from a purity argument, see \cite[§10.5]{Achar2021Perverse}.
    Note that it follows from Braden's theorem \Cref{T:BraThm} and the usual adjunctions that $\conv$ is the right adjoint to
    $\hyp^+$.
\end{proof}
For a simple perverse sheaf $\ic{\oo'}$ in
$\sem(\ep^0,\gdb^0)$, we let $\Lambda_{geo}(\oo')$ be the maximal
$k$ such that $\ic{\oo'}[-k]$ appears in some $\hyp^+(\ic{\oo})$.
We denote the set of such $\oo$ by $\max(\hyp^+,\oo')$.
Moreover, we write
\[\Lambda_{\conv}(\ic{\oo'})\coloneq\max\{i:\ic{\oo}[i]\subseteq\conv(\ic{\oo'}),\ \oo\in\orb_\db\}.\]

\begin{corollary}\label{C:reciprocity}
Let $\oo'\in\orb_{\db^1}\times\ldots\times\orb_{\db^k}$. Then
\[\Lambda_{geo}(\oo')\le\Lambda_{\conv}(\ic{\oo'}).\]
If equality holds, the multiplicity of
$\ic{\oo'}[-\Lambda_{geo}(\oo')]$ in $\hyp^+(\ic{\oo})$ is at most
the multiplicity of $\ic{\oo}[\Lambda_{geo}(\oo')]$ in
$\conv(\ic{\oo'})$, for every $\oo\in\orb_\db$.
\end{corollary}

\begin{proof}
Let
\[k=\Lambda_{geo}(\oo'),\, s=\Lambda_{\conv}(\ic{\oo'}).\]
The multiplicity of $\ic{\oo'}[-k]$ in $\hyp^+(\ic{\oo})$
is bounded above by
\[
\begin{aligned}
&\dim_\CC\hom(\hyp^+(\ic{\oo}),\ic{\oo'}[-k])\\
&\quad=\dim_\CC\hom(\ic{\oo},\conv(\ic{\oo'})[-k]),
\end{aligned}
\]
where the equality follows by adjunction.
Write
\[\conv(\ic{\oo'})\cong\bigoplus_{\oo'',i}\ic{\oo''}[i]^{\oplus n_{\oo'',i}}.\]
The last Hom-space has dimension
\[\sum_{\oo'',i}n_{\oo'',i}\dim_\CC\ext^{i-k}(\ic{\oo},\ic{\oo''}).\]
If $k>s$, all these Ext-groups vanish by \Cref{S:const}.
Choosing $\oo\in\max(\hyp^+,\oo')$ gives a contradiction.
Thus $k\le s$.
If $k=s$, only the terms with $i=k$ can contribute, and their
sum is $n_{\oo,k}$. This proves the multiplicity assertion.
\end{proof}
The following is a consequence of \Cref{C:reciprocity} and \Cref{T:genrog}.
\begin{corollary}\label{C:polecompconv}
Let $\oo_1\in\orb_{\db^1}$, $\oo_2\in\orb_{\db^2}$ and
$\rho\in\cuss$. Write $\pi_i=\Z_\rho(\oo_i)$. Then
\[\Lambda_{geo}(\oo_1\times\oo_2)=\Lambda(\pi_2,\pi_1)-\langle\db^2,\db^1\rangle\]
and
\[\Lambda(\pi_2,\pi_1)\le\Lambda_{\conv}(\ic{\oo_1}\boxtimes\ic{\oo_2})+\langle\db^2,\db^1\rangle.\]
For every $\oo\in\orb_\db$, the multiplicity of $\Z_\rho(\oo)$
in the image of $J_{\pi_2,\pi_1}$ equals the multiplicity of
\[(\ic{\oo_1}\boxtimes\ic{\oo_2})[-\Lambda_{geo}(\oo_1\times\oo_2)]\]
in $\hyp^+(\ic{\oo})$.
If equality holds in the above bound for the pole, this
multiplicity is at most the multiplicity of
\[\ic{\oo}[\Lambda(\pi_2,\pi_1)-\langle\db^2,\db^1\rangle]\]
in $\conv(\ic{\oo_1}\boxtimes\ic{\oo_2})$.
\end{corollary}

\section{Characteristic cycles on Lusztig's characteristic variety}\label{S:charcy}
We use the filtered comparison of \Cref{L:filtered-hyp} to bound the degrees in which an intersection cohomology complex occurs in hyperbolic restriction. Together with \Cref{T:genrog}, this gives bounds on the poles of intertwining operators.
Let $\bX$ be a smooth complex $\bG$-variety, with the induced action on $T^*\bX$. We recall the map
\[\gr^F\colon \mhm(\bX,\bG)\ra \coh(T^*\bX,\bG)\]
of \Cref{S:MHM} and recall the following strictness properties of mixed Hodge modules.
\begin{theorem}[{\cite[§2]{Saito1990}, \cite{Saito1988}, \cite{laumon1985transformations}}]\label{T:strictness}
    Let $f\colon \bX\ra \bY$ be a $\bG$-equivariant morphism of smooth complex varieties and consider the correspondence
    \[\begin{tikzcd}
        T^*\bX&\arrow[l,"df"]\bX\underset{\bY}{\times}T^*\bY\arrow[r,"p"]&T^*\bY.
    \end{tikzcd}\]
    Let $\pi\colon T^*\bX\ra\bX$ be the projection.
    If $f$ is smooth of relative dimension $d$, then for all mixed Hodge modules $M$
    \[\gr^F(f^*M[d])=df_*p^*(\gr^F(M)).\]
    If $f$ is proper, then
    \[\gr^F(\tH^i(f_*M))=\tH^i(\mathbb{R}p_*\mathbb{L}df^*(\gr^F(M)\otimes_{\mathscr{O}_{T^*\bX}}\pi^*\omega_{\bX/\bY})),\]
    where we consider the derived functors $\mathbb{R}p_*$ and $\mathbb{L}df^*$.
    Finally, for $M_1\boxtimes M_2$ a mixed Hodge module on $\bX\times \bY$, \[\gr^F (M_1\boxtimes M_2)=\gr^F(M_1)\boxtimes \gr^F(M_2).\]
\end{theorem}
Note that if $f$ is a closed embedding, then $p$ is a closed embedding and $df$ is smooth. Thus we can work with the underived functors.
For the two chambers of hyperbolic restriction we use the filtered comparison in Appendix~\ref{A:filtered-hyp}.
\subsection{Lusztig's characteristic variety}
Before we describe the situation for $\bX=\ep$ in more detail, we need to recall several definitions. See for example \cite{Lusztig1991Quivers} or \cite{LapMin25}.
Note first that we can identify $\ep^*$ with the moduli space of representations of the opposite quiver of $Q$ with dimension vector $\db$, \textit{i.e.}, endomorphisms of degree $-1$ acting on $\bigoplus_{i\in\ZZ}\CC^{\db_i}$.
The cotangent space of $\ep$ is $T^*\ep=\ep\times \ep^*$,
with natural moment map
\[\mu\colon \ep\times \ep^*\ra \prod_{i\in\ZZ}\mathrm{End}(\CC^{\db_i}), (T,T')\mapsto T\circ T'-T'\circ T.\] Let $\comp_\db$ be the set of irreducible components of the Lagrangian subvariety $\mathbf{\Lambda}(\db)\coloneq\mu^{-1}(0)$.
For $\oo\in\orb_\db$, let $\mathbf{C}_\oo=\overline{T_\oo^*\ep}$ be the closure of its conormal bundle and put
\[\mathbf{C}_\oo^\circ\coloneq \mathbf{C}_\oo\setminus\bigcup_{\oo''\neq\oo}\mathbf{C}_{\oo''}.\]
Then \[\comp_\db=\{\mathbf{C}_\oo:\oo\in\orb_\db\}.\]

Let $\Pi$ be the path algebra of the doubled quiver $\overline{A_+^\infty}$ \[\begin{tikzcd}
    \cdots\arrow[r,bend right]&\bullet\arrow[l,bend right]\arrow[r,bend right,"f_i"']&\bullet\arrow[l,bend right,"e_i"']\arrow[r,bend right,"f_{i+1}"']&\bullet\arrow[l,bend right,"e_{i+1}"']\arrow[r,bend right]&\arrow[l,bend right]\cdots
\end{tikzcd}\] modulo the relations $f_{i}e_i-e_{i+1}f_{i+1}=0$ for all $i\in \ZZ$. 
\begin{lemma}[{\cite[Lemma 1]{crawleyboevey2000exceptional}}]\label{L:cy2}
    For all $\Pi$-modules $x,y$ with dimension vectors $\db,\db'$ we have
    \[\dim_\CC\hom_\Pi(x,y)+\dim_\CC\hom_\Pi(y,x)-(\db,\db')=\dim_\CC\ext_\Pi^1(x,y).\]
\end{lemma}
The variety $\mathbf{\Lambda}(\db)$ is the moduli space of complex $\Pi$-modules with underlying graded vector space $\bigoplus_{i\in\ZZ}\CC^{\db_i}$.
We now come to the following construction of \cite{AizLap}.
For $\oo_i\in\orb_{\db^i}$, $i=1,2$, we set
\[\mathbf{C}_{\oo_1}*\mathbf{C}_{\oo_2}\coloneq \overline{\{x: 0\ra x_2\ra x \ra x_1\ra 0, \,x_i\in \mathbf{C}_{\oo_i}, \dim\ext_\Pi^1(x_1,x_2)\text{ is minimal}\}}.\]
\begin{theorem}[{\cite[Theorem 3.1]{AizLap}}]
    For $\oo_i\in\orb_{\db^i}$, $i=1,2$, we have $\mathbf{C}_{\oo_1}*\mathbf{C}_{\oo_2}\in\comp_\db$.
\end{theorem}
We denote the corresponding orbit by $\oo_1*\oo_2\in\orb_\db$.
\begin{theorem}[{\cite[Theorem 12.3]{Lusztig1991Quivers}}]\label{T:suppcond}
    Let $M\in \mhm(\ep,\gdb)$. Then $\gr^F(M)$ is supported on $\mathbf{\Lambda}(\db)$ and $\siso(M)$ is a union of irreducible components in $\comp_\db$.
\end{theorem}
\begin{theorem}\label{T:conc}
Let $\db=\db^1+\db^2$, let $\oo\in\orb_\db$, and let
$\oo_i\in\orb_{\db^i}$ for $i=1,2$. Suppose that
\[(\ic{\oo_1}\boxtimes\ic{\oo_2})[-n]\text{ appears in }\hyp^+\ic{\oo}.\]
For every $\mathbf{C}_1\in\comp_{\db^1}$ and $\mathbf{C}_2\in\comp_{\db^2}$ with
\[\mathbf{C}_1\subseteq\siso(\ic{\oo_1})\quad\text{and}\quad\mathbf{C}_2\subseteq\siso(\ic{\oo_2}),\]
we have $\mathbf{C}_1\times \mathbf{C}_2\subseteq \siso(\ic{\oo})$ and the bounds
\[\hom(\mathbf{C}_1,\mathbf{C}_2)-\ext^1(\mathbf{C}_1,\mathbf{C}_2) \le\ n+\langle\db^2,\db^1\rangle \le\ \hom(\mathbf{C}_1,\mathbf{C}_2).\]
If equality holds on the right, then $\mathbf{C}_2*\mathbf{C}_1\subseteq\siso(\ic{\oo})$.
\end{theorem}

\begin{proof}
Take $\lambda=(1,0)$ in \Cref{S:hyploc}. Set $Y=T^*\ep$ and write $Y^\pm$ for its attracting and repelling loci. We recall the projections
\[p^\pm\colon Y^\pm\ra T^*\ep^0.\] By the external product formula in \Cref{T:strictness},
\[\mathbf{C}_1\times \mathbf{C}_2 \subseteq \siso(\ic{\oo_1}\boxtimes\ic{\oo_2}).\]
Since $\ic{\oo_1}\boxtimes\ic{\oo_2}$ occurs in
${}^p\!\tH^n(\hyp^+\ic{\oo})$, \Cref{L:filtered-hyp} gives
\[\mathbf{C}_1\times \mathbf{C}_2 \subseteq \siso\bigl({}^p\!\tH^n(\hyp^+\ic{\oo})\bigr) \subseteq T^*\ep^0\cap\siso(\ic{\oo}).\]

As in \Cref{S:MHM}, use the Hodge filtration on the pure Hodge module
lifting $\ic{\oo}$. By \cite[Lemme~5.1.13]{Saito1988},
$\gr^F\ic{\oo}$ is Cohen--Macaulay of pure dimension $\dim\ep$.
Recall that for a Cohen--Macaulay sheaf $\mathcal M$ and a closed
subscheme $\mathbf{Z}\subseteq\bX$, one has
\[\mathscr{E}xt_\bX^q(\mathscr O_{\mathbf{Z}},\mathcal M)=0\]
for $q<\operatorname{codim}_{\supp\mathcal M}
       (\mathbf{Z}\cap\supp\mathcal M)$,
see \cite[Tags 0AVZ and 0BUX]{StacksFilteredComplexes}.
For generic $(x_1,x_2)\in \mathbf{C}_1\times \mathbf{C}_2$, we claim that over the generic point of $\mathbf{C}_1\times \mathbf{C}_2$ we have
\[\mathrm{Tor}^{\mathscr O_{T^*\ep}}_j \bigl(\mathscr O_{Y^\pm},\gr^F\ic{\oo}\bigr)=0\]
if $j>\dim\bigl((p^\pm)^{-1}(x_1,x_2)\cap\siso(\ic{\oo})\bigr)$.
Indeed, the image of $Y^\pm\cap\siso(\ic{\oo})$ under $p^\pm$
lies in $\mathbf{\Lambda}(\db^1)\times\mathbf{\Lambda}(\db^2)$, of which $\mathbf{C}_1\times \mathbf{C}_2$
is an irreducible component of dimension $\dim\ep^0$. Remove from
$T^*\ep^0$ the closures of the images of those irreducible components
of the intersection whose image closures do not contain $\mathbf{C}_1\times \mathbf{C}_2$.
This leaves an open neighbourhood of the generic point of $\mathbf{C}_1\times \mathbf{C}_2$.
For each component meeting this neighbourhood, the closure of its image
is irreducible and contains $\mathbf{C}_1\times \mathbf{C}_2$, so it equals $\mathbf{C}_1\times \mathbf{C}_2$.
By the dimension of fibres theorem, the dimension of each such component
is therefore at most
\[\dim\ep^0+ \dim\bigl((p^\pm)^{-1}(x_1,x_2)\cap\siso(\ic{\oo})\bigr).\]
Writing $c=\dim\ep-\dim\ep^0$, Koszul duality gives
\[\mathrm{Tor}^{\mathscr O_Y}_j \bigl(\mathscr O_{Y^\pm},\gr^F\ic{\oo}\bigr) \simeq \mathscr{E}xt^{c-j}_{\mathscr O_Y} \bigl(\mathscr O_{Y^\pm},\gr^F\ic{\oo}\bigr)\otimes\det N^*_{Y^\pm/Y}.\]
By the preceding dimension bound, the codimension of the intersection in
$\siso(\ic{\oo})$ is at least
$\dim\ep-\dim\ep^0-\dim\bigl((p^\pm)^{-1}(x_1,x_2)\cap\siso(\ic{\oo})\bigr)$.
The Cohen--Macaulay vanishing above therefore gives the claim.

By \Cref{L:filtered-hyp}, applied to $\hyp^+\ic{\oo}$, the
weight-zero part of the twisted Tor pushforward is nonzero at the generic point of
$\mathbf{C}_1\times \mathbf{C}_2$ in Tor degree $\dim(\ep^+/\ep^0)-n$. Hence
\[\dim(\ep^+/\ep^0)-n \leq \dim\bigl((p^-)^{-1}(x_1,x_2)\cap\siso(\ic{\oo})\bigr).\]
By \Cref{T:BraThm} and Verdier self-duality of the intersection
cohomology complexes,
\[\ic{\oo_1}\boxtimes\ic{\oo_2}\text{ occurs in }{}^p\!\tH^{-n}(\hyp^-\ic{\oo}).\]
Applying the same argument to the
inverse cocharacter gives
\[\dim(\ep^-/\ep^0)+n \leq \dim\bigl((p^+)^{-1}(x_1,x_2)\cap\siso(\ic{\oo})\bigr).\]
The fibre
$(p^-)^{-1}(x_1,x_2)\cap\mathbf{\Lambda}(\db)$ consists of the $x$ which occur in extensions
\[0\ra x_2\ra x\ra x_1\ra 0.\]
By the description in \cite[proof of Proposition~9.1]{AizLap},
we have, for generic $(x_1,x_2)\in \mathbf{C}_1\times \mathbf{C}_2$,
\[\dim\bigl((p^-)^{-1}(x_1,x_2)\cap\mathbf{\Lambda}(\db)\bigr) =\dim\bU_\db^+-\hom(\mathbf{C}_1,\mathbf{C}_2)+\ext^1(\mathbf{C}_1,\mathbf{C}_2).\]
Since
\[\dim(\ep^+/\ep^0)=\dim\bU_\db^+-\langle\db^2,\db^1\rangle,\]
this gives the lower bound. Interchanging the two blocks gives the
corresponding fibre dimension for $p^+$ and
$\dim(\ep^-/\ep^0)=\dim\bU_\db^+-\langle\db^1,\db^2\rangle$.
The upper bound follows by \Cref{L:cy2}.

If the upper bound is attained, then
for generic $(x_1,x_2)\in \mathbf{C}_1\times \mathbf{C}_2$ the closed subset
\[(p^+)^{-1}(x_1,x_2)\cap\siso(\ic{\oo}) \subseteq (p^+)^{-1}(x_1,x_2)\cap\mathbf{\Lambda}(\db)\]
has the same dimension as the vector space on the right, so the
inclusion is an equality. These fibres consist of extensions
with submodule $x_1$ and quotient $x_2$. The closure of their
$\gdb$-saturation is $\mathbf{C}_2*\mathbf{C}_1$ by \cite[proof of Theorem~3.1]{AizLap}.
Since $\siso(\ic{\oo})$ is closed and $\gdb$-stable, it contains
$\mathbf{C}_2*\mathbf{C}_1$.
\end{proof}
We collect the results of \Cref{S:JanGeo} and \Cref{S:charcy} in the following theorem.

\begin{theorem}\label{T:polesmicrolocal}
Let $\rho\in\cuss$, $\pi_i=\Z_\rho(\oo_i)$ and $\mathbf{C}_i=\mathbf{C}_{\oo_i}$ for $i=1,2$.
\begin{enumerate}
\item The order of the pole satisfies
\[\hom(\mathbf{C}_1,\mathbf{C}_2)-\ext^1(\mathbf{C}_1,\mathbf{C}_2) \le\Lambda(\pi_2,\pi_1)\le\hom(\mathbf{C}_1,\mathbf{C}_2).\]
In particular,
\[0\le\Lambda(\pi_2,\pi_1)+\Lambda(\pi_1,\pi_2)-(\db^1,\db^2) \le\ext^1(\mathbf{C}_1,\mathbf{C}_2).\]
\item If $\ext^1(\mathbf{C}_1,\mathbf{C}_2)=0$, then
$\Lambda(\pi_2,\pi_1)=\hom(\mathbf{C}_1,\mathbf{C}_2)$,
$\Lambda(\pi_1,\pi_2)=\hom(\mathbf{C}_2,\mathbf{C}_1)$, and
$J_{\pi_2,\pi_1}$ is an isomorphism.
\item The image of $J_{\pi_2,\pi_1}$ is semisimple.
\item If $\Lambda(\pi_2,\pi_1)=\hom(\mathbf{C}_1,\mathbf{C}_2)$, then for every
constituent $\Z_\rho(\oo)$ of the image of $J_{\pi_2,\pi_1}$ one has
\[\mathbf{C}_2*\mathbf{C}_1\subseteq\siso(\ic{\oo}).\]
\end{enumerate}
\end{theorem}

\begin{proof}
By \Cref{T:genrog} and the description of the top filtration quotient
in \Cref{L:filtpol}, the multiplicity of $\Z_\rho(\oo)$ in the image of
$J_{\pi_2,\pi_1}$ equals the multiplicity of
\[(\ic{\oo_1}\boxtimes\ic{\oo_2})[-\Lambda_{geo}(\oo_1\times\oo_2)] \quad\text{in}\quad \hyp^+(\ic{\oo}),\]
where
\[\Lambda_{geo}(\oo_1\times\oo_2)= \Lambda(\pi_2,\pi_1)-\langle\db^2,\db^1\rangle .\]
Applying \Cref{T:conc} to any constituent of the image
gives the two bounds on $\Lambda(\pi_2,\pi_1)$ and, at the upper bound,
the inclusion in (4). Adding the two upper bounds obtained by exchanging
the blocks and using \Cref{L:cy2} gives the upper bound on the sum in (1).
Its lower bound follows from \Cref{L:intiso,L:alpha}. If
$\ext^1(\mathbf{C}_1,\mathbf{C}_2)=0$, the pole orders equal the corresponding Hom
dimensions. Their sum is $(\db^1,\db^2)$ by \Cref{L:cy2}, so
\Cref{L:intiso} gives (2). Assertion (3) is \Cref{C:semisimple}.
\end{proof}
Lastly, we perform a sanity check as in \cite{Dro25} and use the language of the theorem.
We recall the Aubert--Zelevinsky involution of \cite{Aubert1995Dualite,MoeglinWaldspurger1986Zelevinski}.
\[(-)^*\colon \irr_{\db,\rho}\ra\irr_{\db,\rho}.\]
By \cite[Lemma 4.8]{Dro25},
\[\Lambda(\pi_1,\pi_2)=\Lambda(\pi_2^*,\pi_1^*).\]
For $\oo\in\orb_\db$, define $\oo^*\in\orb_\db$ by \[\Z_\rho(\oo^*)=\Z_\rho(\oo)^*.\]
Similarly, for $\mathbf{C}_\oo\in\comp_\db$ we write $\mathbf{C}_\oo^*\coloneq \mathbf{C}_{\oo^*}$.
We extend this notation to unions of irreducible components of $\mathbf{\Lambda}(\db)$. 
Using this notation, we recall from \cite[Theorem 5.5.5]{KashiwaraSchapira1990Sheaves} and \cite[Proposition 7.2]{EvensMirkovic1997Fourier} that \[\siso(\ic{\oo_i^*})=\siso(\ic{\oo_i})^*.\]
Duality of preprojective modules exchanges
$\hom(\mathbf{C}_1,\mathbf{C}_2)$ with $\hom(\mathbf{C}_2^*,\mathbf{C}_1^*)$ and preserves the generic
$\ext^1$ dimension, as in \cite[p.~17]{Dro25}. Thus the two bounds in (1)
agree for $(\pi_1,\pi_2)$ and $(\pi_2^*,\pi_1^*)$.

\appendix
\newtheorem{appprop}[lemma]{Proposition}
\section{The type \texorpdfstring{$A$}{A} intertwining operator and hyperbolic restriction}
\label{app:intertwiner}

In this appendix we generalize the results of \cite[Proposition~5.5 and Theorem~5.6]{FujitaHernandez2026} and
\cite[Lemma~3.1]{Ciubotaru2026Jantzen} to the setting at hand. At their core, the arguments are an application of the Hard-Lefschetz principle. The main difference is that unlike in the settings of the above papers, we have to deal with both attracting and repelling loci at the same time. We use the Fourier-Deligne transform to transform the repelling locus into an attracting locus and hence to use the above strategy.

Fix $\db=\db^1+\db^2$, $\alpha=(|\db^1|,|\db^2|)$, and take
$\lambda=(1,0)$ in \Cref{S:JanGeo}.
The weight decomposition for $\chi_\lambda$ is
$\ep=\ep^0\oplus \mathbf{V}_+\oplus \mathbf{V}_-$, with $\ep^\pm=\ep^0\oplus \mathbf{V}_\pm$.

\subsection{Geometric construction of intertwining operators}\label{app:fixed-locus-comparison}

Let $s_\pm\colon\ep^0\hra\ep^\pm$ be the zero sections,
$j_\pm$ the complementary open immersions, and $i=i_\pm s_\pm$.
Contraction gives $\hyp^\pm\cf\simeq s_\pm^!i_\pm^*\cf$
\cite[\S3, Lemma~6]{braden2003hyperbolic}.
Adjunction gives $I_\pm(\cf)\colon\hyp^\pm\cf\to i^*\cf$. For the following construction, see \cite[\S1.3 and Corollary~3.1]{KhanRaviConcentration}.

Let $s$ be the first Chern class of the weight-one character of $\gm$.
It generates
\[H^2_{\gm}(*)= \hom_{D_c^b(*,\gm)}(\uc_*,\uc_*[2]).\]
For any $\gm$-variety $\bX$ and $\cf,\mathcal G\in D_c^b(\bX,\gm)$, we write
\[\hom_\gm^\bullet(\cf,\mathcal G) =\bigoplus_{k\in\ZZ}\hom_{D_c^b(\bX,\gm)}(\cf,\mathcal G[k]).\]
For $p\colon \bX\to *$, a class $\alpha\in H_\gm^k(*)$ acts on $\cf$
by the morphism $p^*\alpha\otimes\mathrm{id}_{\cf}\colon
\cf\to\cf[k]$. This makes the graded morphism space a
module over $H_\gm^\bullet(*)=\CC[s]$. After inverting $s$, a homogeneous
fraction $f/s^n$ has degree $\deg f-2n$. An outer superscript $0$ denotes
the degree-zero part.

We have the Euler class
\[e(W)=z^*\operatorname{Th}(W)=c_{\operatorname{rank}W}^{\gdb^0\times\gm}(W) \in\hom_{D_c^b(\bX,\gdb^0\times\gm)} (\uc_{\bX},\uc_{\bX}[2\operatorname{rank}W]).\]
The bundle on $\ep^\pm=\ep^0\times \mathbf{V}_\pm$ with fibre $\mathbf{V}_\pm$ has the
section $(x,v)\mapsto v$. On $\ep^\pm\setminus\ep^0$ it spans a trivial
equivariant line, so $j_\pm^*e(\mathbf{V}_\pm)=0$ by the Whitney sum formula.
The localization triangle in $D_c^b(\ep^\pm,\gdb^0\times\gm)$ is
\[s_{\pm*}s_\pm^!i_\pm^*\cf \ra i_\pm^*\cf \ra j_{\pm*}j_\pm^*i_\pm^*\cf \ra (s_{\pm*}s_\pm^!i_\pm^*\cf)[1].\]
For any $\mathcal G\in D_c^b(\ep^\pm\setminus\ep^0,\gdb^0\times\gm)$,
the projection formula gives
\[e(\mathbf{V}_\pm)\cdot\mathrm{id}_{j_{\pm*}\mathcal G} = j_{\pm*}\bigl(j_\pm^*e(\mathbf{V}_\pm)\otimes\mathrm{id}_{\mathcal G}\bigr)=0.\]
Taking $\mathcal G=j_\pm^*i_\pm^*\cf$, we see that the identity of the
third term vanishes after inverting $e(\mathbf{V}_\pm)$. The first arrow and
its pullback $I_\pm(\cf)$ by $s_\pm^*$ are therefore isomorphisms.

Inverting both Euler classes thus defines the $\gdb^0\times\gm$-equivariant
isomorphism $R_{\mathrm{hyp}}(\cf)=I_-(\cf)^{-1}\circ I_+(\cf)$.
Restriction to $\gm$ sends $e(\mathbf{V}_\pm)$ to $(\pm s)^{\dim \mathbf{V}_\pm}$ and gives
\[\left.R_{\mathrm{hyp}}(\cf)\right|_{\gm} \in\bigl(\hom^\bullet_{\gm} (\hyp^+\cf,\hyp^-\cf)[s^{-1}]\bigr)^0.\]

We use the Springer sheaf $\cS_\db$ with the shift
$[\dim\tep^{\boldsymbol i}]$ on the summand for each word $\boldsymbol i$,
reading the flag quotients from the top \cite[\S\S1.4--1.6, 2.2]{Lusztig1991Quivers}.
Put
$\cS_\db^0=\cS_{\db^1}\boxtimes\cS_{\db^2}$ and use Yoneda composition in
\[A=\ext^\bullet_{\gdb}(\cS_\db,\cS_\db),\, B=\ext^\bullet_{\gdb^0}(\cS_\db^0,\cS_\db^0).\]
Give $\cS_\db^0$ the trivial additional $\gm$-action. For
$\fm_j\in\mul_{\db^j}$ and $\mathcal Y\in D_c^b(\ep^0,\gdb^0\times\gm)$, set
\[C_{\fm_1,\fm_2}(\mathcal Y)= \hom^\bullet_{B[s]}\!\left( \ext^\bullet_{\gdb^0\times\gm}(\mathcal Y,\cS_\db^0), (L_{\fm_1}\otimes L_{\fm_2})[s]\right).\]
The two lattices are $\mathcal L_\pm=C_{\fm_1,\fm_2}(\hyp^\pm\cS_\db)$.

With $\lambda=(1,0)$ and $\overline\lambda=(0,1)$, the underlying Hecke
modules of the $s$-adic completions are
\[\widehat{\mathcal L}_-\simeq(L_{\fm_1}\times L_{\fm_2})[[s]]_\lambda,\, \widehat{\mathcal L}_+\simeq(L_{\fm_2}\times L_{\fm_1})[[s]]_{\overline\lambda}.\]
The hyperbolic comparison induces the $A$-linear isomorphism
\[C_{\fm_1,\fm_2}\bigl(R_{\mathrm{hyp}}(\cS_\db)\bigr)\colon \mathcal L_+[s^{-1}]\xrightarrow{\sim}\mathcal L_-[s^{-1}].\]

By adjunction and flag concatenation \cite[\S4.6, (4.6)]{VaragnoloVasserot2011Canonical},
$\mathcal L_\pm$ is block coinduction with the stable block last
(block $2$ for $\hyp^-$ and block $1$ for $\hyp^+$).
Subtracting the shift $[2\dim \mathbf{V}_\pm]$ of $(p^\pm)^!$ from the concatenation
shift gives $[\dim\bU_\db^+-\dim \mathbf{V}_\pm]$ on $\mathcal L_\pm$.

The shuffle basis gives a finite $\CC[s]$-basis of each lattice
\cite[\S2.2, Remark~2.1]{LaudaVazirani2011}. The Hecke equivalence
preserves block restriction and hence its right adjoint, coinduction
\cite[\S2.1 and Proposition~3.2]{Gurevich2023Simple}.
On $\cS_\db$ the $\gdb^0\times\gm$-action is pulled back along
$(g_1,g_2,t)\mapsto\mathrm{diag}(tg_1,g_2)$. Each first-block flag quotient thus acquires the weight-one character, so the above adjunction translates its Chern root by $s$ in the coinducing module and fixes the second-block roots.
In Webster's exponential coordinates the translation by $s$ multiplies
the corresponding Bernstein generator by $e^s$
\cite[Theorem~3.10]{Webster2020Hecke}. 

Over $\CC(s)$ the two coinductions are irreducible, so any nonzero
intertwiner between them is invertible. For a homogeneous intertwiner,
each nonzero matrix entry in homogeneous bases is a scalar multiple
of a power of $s$. Its nonzero determinant has the same form, so its
inverse is also defined over $\CC[s,s^{-1}]$.

\subsection{Comparison of intertwining operators}

\begin{appprop}\label[proposition]{prop:intertwiner-hyp}
Let $\pi_j=\Z(\fm_{j,\rho})$ for $j=1,2$. Choose $\kappa\in\CC^\times$
so that under the type equivalence of \Cref{S:Hecke} the twist
$\lvert-\rvert^{\kappa s}$ multiplies the first block's Bernstein
generators by $e^s$. As maps from $\widehat{\mathcal L}_-[s^{-1}]$
to $\widehat{\mathcal L}_+[s^{-1}]$, we have
\[\ct_{|\db|,\rho}\bigl(J_{\pi_1,\pi_2,1,0}(\kappa s)\bigr) =v(s)\,s^{-\langle\db^2,\db^1\rangle} C_{\fm_1,\fm_2}\bigl(R_{\mathrm{hyp}}(\cS_\db)\bigr)^{-1},\]
for some $v(s)\in\CC[[s]]^\times$.
\end{appprop}

\begin{proof}
For $\alpha=(|\db^1|,|\db^2|)$, let $w_\alpha$ exchange the two blocks.
In the coinduction convention the regular quiver Hecke intertwiner
$R_{\mathrm{KLR}}\colon\mathcal L_-\to\mathcal L_+$ is given by
precomposition with left multiplication by $\varphi_{w_\alpha^{-1}}$
and interchange of the coefficient factors
\cite[\S1.2.1, (1.5)--(1.7)]{KKKO2015Simplicity}. Its unshifted degree is
$\dim \mathbf{V}_++\dim \mathbf{V}_-$, one for each cross-block pair of adjacent colors.
The lattice shifts add $\dim \mathbf{V}_+-\dim \mathbf{V}_-$, so $R_{\mathrm{KLR}}$
raises degree by $2\dim \mathbf{V}_+$.

By \Cref{app:fixed-locus-comparison} the hyperbolic comparison has
degree zero. Since $\deg s=2$, the generically invertible endomorphism
$s^{-\dim \mathbf{V}_+}C_{\fm_1,\fm_2}(R_{\mathrm{hyp}}(\cS_\db))\circ R_{\mathrm{KLR}}$
preserves each homogeneous piece of $\mathcal L_-[s^{-1}]$. Each nonzero piece is finite-dimensional, so this endomorphism has a complex eigenvector on it.
This vector generates the generic fibre by irreducibility, so for its
eigenvalue $c\in\CC^\times$ we have
\[R_{\mathrm{KLR}}=c\,s^{\dim \mathbf{V}_+} C_{\fm_1,\fm_2}\bigl(R_{\mathrm{hyp}}(\cS_\db)\bigr)^{-1}.\]

Embed each $\pi_j$ in a product of integer twists of $\rho$. Both block
intertwiners factor into pairwise exchanges and restrict by naturality
\cite[Proposition~7.8 and Lemma~7.12]{Dat05}. For a pair
$\rho\otimes\lvert\det\rvert^i$, $\rho\otimes\lvert\det\rvert^j$
with $i,j\in\ZZ$, the analytic intertwiner has Laurent order
$-\delta_{ij}$ at $s=0$ \cite[Lemma~4.2]{Dro25}. The regular quiver Hecke crossing has order zero
since its square is one if $i=j$ and it is a shuffle basis element if
$i\ne j$. Thus the analytic intertwiner is $s^{-\delta_{ij}}$ times
the regular crossing up to a unit. The twist
$\rho\otimes\lvert\det\rvert^i$ occurs $\db_i^1$ times in the first
block and $\db_i^2$ times in the second, so the total exponent is
$\dim \mathbf{V}_+-\sum_i\db_i^1\db_i^2
=\dim \mathbf{V}_+-\dim\bU_\db^+=-\langle\db^2,\db^1\rangle$.
\end{proof}

\subsection{Hard Lefschetz and the Jantzen filtration}
\label{app:smith}

Let $\mathcal Y$ be a direct sum of shifted simple perverse sheaves on $\ep^0$. On the summands
of the Springer sheaf $\cS_\db^0$, Yoneda composition and finite-projective
duality \cite[Dual Basis Lemma~2.9]{LamModulesRings} give
\[C_{\fm_1,\fm_2}(\mathcal Y)\simeq \ext^\bullet_{\gdb^0\times\gm}(\cS_\db^0,\mathcal Y) \otimes_B(L_{\fm_1}\otimes L_{\fm_2}).\]
As in the construction of the two lattices, each word summand of the
Springer sheaf is shifted by the dimension of its flag resolution.
With this grading, each summand of $\mathcal Y$ of the form
\[\bigl(\ic{\oo(\fm_1)}\boxtimes\ic{\oo(\fm_2)}\bigr)[-n]\]
contributes one free $\CC[s]$-generator in degree $n$.

\begin{lemma}\label{L:smith-hyp}
Let $\cf=\ic{\oo}$ for $\oo\in\orb_\db$. There are homogeneous
$\CC[s]$-bases $(u_j)$ of $C_{\fm_1,\fm_2}(\hyp^+\cf)$ and $(v_j)$
of $C_{\fm_1,\fm_2}(\hyp^-\cf)$ such that, after inverting $s$,
\[C_{\fm_1,\fm_2}\bigl(R_{\mathrm{hyp}}(\cf)\bigr)(u_j)=s^{\deg u_j}v_j,\,\deg v_j=-\deg u_j.\]
For each $n\in\ZZ$, the number of $u_j$ of degree $n$ is the multiplicity of
$\ic{\oo(\fm_1)}\boxtimes\ic{\oo(\fm_2)}$ in ${}^p\!\tH^n(\hyp^+\cf)$.
\end{lemma}

\begin{proof}
Work first over the algebraic closure of a finite field with
$\ql$-coefficients, for an odd prime $\ell$ invertible in the base field.
We use pure Weil structures for the decomposition theorem and relative
hard Lefschetz. We will suppress Tate twists.

Suppose first that $\gm$ acts on $\ep^0\times \mathbf{W}$ with weight one on $\mathbf{W}$.
Let $\mathcal G$ be a pure $\gdb^0\times\gm$-equivariant perverse sheaf, $i$ the zero
section, and $j$ the complementary open immersion. We follow
\cite[Proposition~5.5 and Theorem~5.6]{FujitaHernandez2026} and
\cite[Lemma~3.1]{Ciubotaru2026Jantzen}, where the fixed locus is a point.
The localization triangle on $\ep^0$ is
\[i^!\mathcal G\ra i^*\mathcal G \ra i^*j_*j^*\mathcal G\ra i^!\mathcal G[1].\]
Applying $\ext^\bullet_{\gdb^0\times\gm}(\cS_\db^0,-)$ gives a long exact sequence.
We compute the third term and show that it is the cokernel of the first map on Ext.
The claim then follows in this case from Hard Lefschetz as in the above references.

If $\mathbf{W}=0$, the first arrow is the identity and every generator has degree
zero. Otherwise let $q\colon\ep^0\times(\mathbf{W}\setminus\{0\})\to
\ep^0\times\mathbf{P}(\mathbf{W})$ be the quotient map and $p$ the projection to $\ep^0$.
Perverse descent along the one-dimensional torsor $q$ gives
$j^*\mathcal G=q^*\overline{\mathcal G}[1]$ for a pure perverse sheaf
$\overline{\mathcal G}$. Contraction identifies $i^*j_*j^*\mathcal G$ with
$(p q)_*j^*\mathcal G$ \cite[Lemma~6]{braden2003hyperbolic}. By adjunction and descent along $q$, the third Ext term is
\[\ext^\bullet_{\gdb^0}(p^*\cS_\db^0,\overline{\mathcal G})[1]\simeq\ext^\bullet_{\gdb^0}(\cS_\db^0,p_*\overline{\mathcal G})[1].\]
On $\ep^0\times\mathbf{P}(\mathbf{W})$, write $\mathscr{O}(1)=\mathrm{pr}_2^*\mathscr{O}_{\mathbf{P}(\mathbf{W})}(1)$, where $\mathrm{pr}_2$ is the second projection. 
Put $h=c_1^{\gdb^0}(\mathscr{O}(1))$. Under this identification, $s$ acts on a morphism
$f\colon\cS_\db^0\ra p_*\overline{\mathcal G}[k]$ by
\[s f=p_*(h\otimes\mathrm{id}_{\overline{\mathcal G}})[k]\circ f.\]

By purity and the decomposition theorem, $i^!\mathcal G$, $i^*\mathcal G$ and $p_*\overline{\mathcal G}$ split geometrically into shifts of simple summands of $\cS_\db^0$
\cite[Theorem~8]{braden2003hyperbolic}, \cite[Theorems~3.11--3.12]{Sun2012Decomposition}.
Thus $\ext^\bullet_{\gdb^0\times\gm}(\cS_\db^0,i^!\mathcal G)$ is projective over $B[s]$.
The first map on Ext becomes invertible after inverting $e(\mathbf{W})$ by \Cref{app:fixed-locus-comparison}, so it is injective since $e(\mathbf{W})$ is monic in $s$.
The third term is projective over $B$ by the same decomposition, hence the short exact sequence splits over $B$. It follows that the coefficient map
$C_{\fm_1,\fm_2}(i^!\mathcal G)\to C_{\fm_1,\fm_2}(i^*\mathcal G)$
has cokernel
\[\bigl(\ext^\bullet_{\gdb^0}(\cS_\db^0,p_*\overline{\mathcal G}) \otimes_B(L_{\fm_1}\otimes L_{\fm_2})\bigr)[1].\]

Write $P_j={}^p\!\tH^j(p_*\overline{\mathcal G})$ and choose a decomposition $p_*\overline{\mathcal G}\simeq\bigoplus_jP_j[-j]$.
With respect to this decomposition, $s$ has components
\[h_{rj}\in\hom_{D_c^b(\ep^0,\gdb^0)}(P_j[-j],P_r[2-r])=\ext^{j+2-r}_{\gdb^0}(P_j,P_r).\]
The coefficient of the unshifted perverse sheaf $P_j$ is
\[\ext^\bullet_{\gdb^0}(\cS_\db^0,P_j)\otimes_B(L_{\fm_1}\otimes L_{\fm_2})\simeq\hom_{\perv(\ep^0,\gdb^0)}(\ic{\oo(\fm_1)}\boxtimes\ic{\oo(\fm_2)},P_j).\]
The map induced by $h_{rj}$ from the multiplicity space of $\ic{\oo(\fm_1)}\boxtimes\ic{\oo(\fm_2)}$ in $P_j$ to its multiplicity space in $P_r$ has degree $j+2-r$. Both spaces are concentrated in degree zero, so this map vanishes unless $r=j+2$. Before the shift $[1]$, $s$ therefore acts by
\[\hom(S,P_j)\ra\hom(S,P_{j+2}),\, f\mapsto h_{j+2,j}\circ f,\]
where $S=\ic{\oo(\fm_1)}\boxtimes\ic{\oo(\fm_2)}$.
By relative hard Lefschetz for $p$ and $\mathscr{O}(1)$
\cite[Theorem~5.4.10]{BBD1982}, a Lefschetz string of length $m$ occupies degrees
$1-m,3-m,\ldots,m-1$. After the shift $[1]$, it contributes a generator $v$
of degree $-m$ to the cokernel, with relation $s^m v=0$ in degree $m$.
A unit elementary divisor pairs generators of equal degree. Since
$i^!\mathcal G\in{}^pD^{\geq 0}$ and $i^*\mathcal G\in{}^pD^{\leq 0}$, this degree is zero.
The claim thus follows.

For general $\cf=\ic{\oo}$, take its relative Fourier--Deligne transform
in $\mathbf{V}_-$, normalized by $[\dim \mathbf{V}_-]$. The result is a pure perverse sheaf
$\mathcal G$ on $\ep^0\times(\mathbf{V}_+\oplus \mathbf{V}_-^*)$, with weight one on
both moving summands \cite[\S\S1.2--1.3]{Laumon1987Fourier}. For the zero
section $i'$, evaluation of the Fourier kernel at zero and then duality
with \Cref{T:BraThm} give
\[i'^*\mathcal G\simeq(p^-)_!i_-^*\cf[\dim \mathbf{V}_-]\simeq\hyp^-\cf[\dim \mathbf{V}_-],\,i'^!\mathcal G\simeq\hyp^+\cf[-\dim \mathbf{V}_-].\]

Let $i\colon\ep^0\hra\ep$ be the inclusion. As in \Cref{app:fixed-locus-comparison}, $\cf\ra i_*i^*\cf$ becomes an isomorphism after inverting $e(\mathbf{V}_+\oplus\mathbf{V}_-)$.
For the complex $i_*i^*\cf$ supported on $\ep^0$, the adjunction morphisms
\[I_\pm(i_*i^*\cf)\colon\hyp^\pm(i_*i^*\cf)\ra i^*i_*i^*\cf\]
of \Cref{app:fixed-locus-comparison} both identify with $\mathrm{id}_{i^*\cf}$ in $D_c^b(\ep^0,\gdb^0\times\gm,\ql)$.
Let $k\colon\ep^0\times\mathbf{V}_-^*\hra\ep^0\times(\mathbf{V}_+\oplus\mathbf{V}_-^*)$ send $(x,\xi)$ to $(x,0,\xi)$, and let $\mathrm{pr}_1$ denote the projection to $\ep^0$. Since the Fourier pairing vanishes on the support of $i_*i^*\cf$, \cite[Definition~1.2.1.1]{Laumon1987Fourier} gives that the Fourier transform of $i_*i^*\cf$ is
\[k_*\mathrm{pr}_1^*(i^*\cf)[\dim\mathbf{V}_-].\]
Write $i'=kz$, where $z$ is the zero section of $\mathrm{pr}_1$. On this transform, the self-intersection formula identifies $i'^!\ra i'^*$ with
\[i^*\cf[-\dim\mathbf{V}_-]\xrightarrow{(e(\mathbf{V}_-^*)\otimes\mathrm{id}_{i^*\cf})[-\dim\mathbf{V}_-]}i^*\cf[\dim\mathbf{V}_-].\]
By naturality, the localized map $i'^!\mathcal G\ra i'^*\mathcal G$, shifted by $[\dim\mathbf{V}_-]$, is therefore
\[\hyp^+\cf\xrightarrow{R_{\mathrm{hyp}}(\cf)}\hyp^-\cf\xrightarrow{e(\mathbf{V}_-^*)\otimes\mathrm{id}_{\hyp^-\cf}}\hyp^-\cf[2\dim\mathbf{V}_-].\]

On the chosen coefficient, the Euler class of $\mathbf{V}_-^*$ acts by
$s^{\dim \mathbf{V}_-}$. The claim then follows from the purely attracting statement.
After spreading out the orbit stratifications and clearing the Euler denominators, generic base change and comparison identify the morphisms defining the coefficient matrices in sufficiently large characteristic with their counterparts over $\CC$, preserving $s$ \cite[Theorem~2.5, \S\S3.4--3.5 and 4.1--4.2]{Sun2017Comparison}.
\end{proof}
\section{Hyperbolic restriction and singular support}\label{A:filtered-hyp}
In this section we give a technical lemma how hyperbolic restriction behaves with respect to taking characteristic cycles. Note that \emph{a priori}, such a compatibility is not expected, since hyperbolic restriction is the composition of proper pullback with smooth pushforward. However, as in many other instances, in the setting of hyperbolic restriction, the functor behaves like it is a composition of proper pushforwards and smooth pullbacks, and thus allows the following description.

Let $\bX$ be a complex vector space with a linear $\gm$-action of
weights $0,1,-1$. We use the correspondence
$\bX^0\xleftarrow{p^\pm}\bX^\pm\xrightarrow{i_\pm}\bX$ of \Cref{S:hyploc}. Write $(z,x,y)$ for coordinates in the fixed,
positive and negative directions, respectively, and put
$a=\dim\bX^+-\dim\bX^0$, $b=\dim\bX^--\dim\bX^0$. Thus
$t(z,x,y)=(z,tx,t^{-1}y)$. Write $(\zeta,\xi,\eta)$ for the
corresponding cotangent coordinates. Put $Y=T^*\bX$ and let $Y^-$ be its repelling locus. We use the same notation for
the maps in the repelling diagram of the cotangent action,
\[T^*\bX^0\xleftarrow{p^-}Y^-=\{x=\eta=0\} \xrightarrow{i_-}T^*\bX.\]
The fixed subspace $T^*\bX^0\subseteq T^*\bX$ is given by
$x=y=\xi=\eta=0$. In the binary quiver correspondence we may take the
cocharacter with slopes $(1,0)$, which gives the same two chambers.

Let $(\mathcal M,F)$ be a strongly $\gm$-equivariant coherent left
$\mathscr D_\bX$-module with an invariant good filtration for the
order filtration on $\mathscr D_\bX$. Its characteristic variety is
$\siso(\mathcal M)=\supp\gr^F\mathcal M$. The repelling description
in \Cref{T:BraThm} is realized on $\mathscr D$-modules by
\[\hyp^+\mathcal M=(p^-)_+\,\mathbb L i_{-,\mathscr D}^*\mathcal M[-a],\]
where $\mathbb L i_{-,\mathscr D}^*$ denotes the unshifted derived
transfer-module inverse image.

We use the convention that the Euler vector field is
\[\theta=\sum_{i=1}^a x_i\partial_{x_i} -\sum_{j=1}^b y_j\partial_{y_j}.\]
Strong equivariance means that the infinitesimal $\gm$-action on
$\Gamma(\bX,\mathcal M)$ agrees with the action of $\theta$. Thus a
section on which $t$ acts by $t^r$ satisfies $\theta m=rm$, and we
call $r$ its Euler weight. The algebraic $\gm$-action decomposes
$\Gamma(\bX,\mathcal M)$ into these integral weight spaces and
preserves $F$. On $T^*\bX$, the coordinates $x_i,\eta_j$ have Euler
weight $1$, and $y_j,\xi_i$ have Euler weight $-1$.

The form $dy_1\wedge\cdots\wedge dy_b$ generates the relative
canonical line $\omega_{\bX^-/\bX^0}$. Write $\omega_0$ for its
$\CC$-span. Since $t^*dy_j=t^{-1}dy_j$, this generator has Euler
weight $-b$. We place $\omega_0$ in filtration degree zero, so
tensoring with $\omega_0$ leaves the filtration index unchanged.
On $Y^-$ we pull back $\omega_{\bX^-/\bX^0}$ from $\bX^-$.
The subscript $0$ denotes the weight-zero part for the action on the
whole tensor product. In particular, tensoring a section of weight
$r$ with the displayed form gives a section of weight $r-b$.

\begin{lemma}\label{L:filtered-hyp}
The complex
\[\left[(p^-)_*\left(\mathbb L i_-^*(\gr^F\mathcal M) \otimes\omega_{\bX^-/\bX^0}\right)\right]_0[-a]\]
has bounded coherent cohomology on $T^*\bX^0$. For each $n$, the
module $\tH^n(\hyp^+\mathcal M)$ admits a good filtration $G$ whose
associated graded is a subquotient of the $n$th cohomology of this
complex. In particular, $\siso(\tH^n(\hyp^+\mathcal M))$ is contained in the support of
\[\left[(p^-)_*\left(\mathrm{Tor}^{\mathscr O_{T^*\bX}}_{a-n}(\mathscr O_{Y^-},\gr^F\mathcal M)\otimes\omega_{\bX^-/\bX^0}\right)\right]_0,\]
which lies in $T^*\bX^0\cap\siso(\mathcal M)$. We set negative Tor groups equal to zero. Moreover,
\[\tH^n(\hyp^+\mathcal M)=0\quad\text{unless }-b\leq n\leq a.\]
The assertions for $\hyp^-$ follow by reversing the cocharacter.
\end{lemma}

\begin{proof}
Work with global sections. The transfer resolution for $i_-$ is the
Koszul complex on $x_1,\ldots,x_a$ in degrees $-a,\ldots,0$. The
relative de Rham complex for $p^-$ uses differentiation in the
$y$-directions in degrees $-b,\ldots,0$. We write $\partial_{y_j}$
for differentiation with respect to $y_j$ and $\eta_j$ for its
principal symbol on $T^*\bX$. Let $e_i$ be the exterior generator
corresponding to $x_i$ and $f_j$ the one corresponding to
$\partial_{y_j}$. Writing relative forms using $\omega_0$ and
including the shift $[-a]$ gives the total complex
\[K^{a-k}=\Gamma(\bX,\mathcal M)\otimes \bigwedge^k\CC^{a+b}\otimes\omega_0, \, d=\sum_i x_i\iota_{e_i^*}+\sum_j\partial_{y_j}\iota_{f_j^*},\]
where $\iota$ denotes contraction with the indicated dual basis
vector.

Give each $e_i$ and each $f_j$ Euler weight $1$. Both $x_i$ and
$\partial_{y_j}$ raise Euler weight by $1$, so $d$ preserves total
weight. In exterior degree $k$, a section of $\mathcal M$ of weight
$r$ therefore contributes to total weight $r+k-b$. Put
\[h=\sum_i\partial_{x_i}(e_i\wedge-)-\sum_j y_j(f_j\wedge-).\]
The Weyl algebra relations and the identity between contraction and exterior multiplication give
\[dh+hd=\theta+k-b\,\text{on }K^{a-k},\]
where $\theta$ acts on the $\mathcal M$ factor
(\textit{cf.}\ \cite[proof of Theorem~3.1]{ChenDirks2023}). The operator $h$
preserves total weight. On the subcomplex of total weight $w$, the
right-hand side acts by $w$, so for $w\neq0$ the map $h/w$ is a
contracting homotopy. Since $K$ is the direct sum of its weight
subcomplexes, the inclusion $K_0\hookrightarrow K$ is a
quasi-isomorphism.

On the summand indexed by the exterior monomial $e_I\wedge f_J$, set
\[F_\ell K^{a-k}\big|_{e_I\wedge f_J} = \Gamma(\bX,F_{\ell-|J|}\mathcal M)\otimes \CC(e_I\wedge f_J)\otimes\omega_0, \, |I|+|J|=k.\]
Multiplication by $x_i$ preserves the filtration on $\mathcal M$.
Differentiation by $\partial_{y_j}$ raises its index by at most one
and removes one $f_j$, so $d$ preserves the displayed filtration.
On passing to the associated graded, the map
$\gr^F_p\mathcal M\to\gr^F_{p+1}\mathcal M$ induced by
$\partial_{y_j}$ is multiplication by the cotangent coordinate
$\eta_j$. Multiplication by $x_i$ still acts by $x_i$. Thus
\[\gr^F d=\sum_i x_i\iota_{e_i^*}+\sum_j\eta_j\iota_{f_j^*}.\]

The sequence $x_1,\ldots,x_a,\eta_1,\ldots,\eta_b$ is a regular
sequence in $\mathscr O_{T^*\bX}$ defining $Y^-$. Its Koszul
complex, with exterior degree $k$ in cohomological degree $-k$,
resolves $\mathscr O_{Y^-}$. Tensoring this resolution with
$\gr^F\mathcal M$ therefore computes $\mathbb L i_-^*(\gr^F\mathcal M)$.
The differential is precisely the one displayed above, and the shift
$[-a]$ places exterior degree $k$ in degree $a-k$. Since $p^-$ is
affine, its pushforward on quasi-coherent sheaves is exact. Taking
weight zero is also exact and commutes with passage to the associated
graded because $F$ is invariant. Hence, on $T^*\bX^0$, we obtain
\[\gr^F K_0\simeq \left[(p^-)_*\left(\mathbb L i_-^*(\gr^F\mathcal M) \otimes\omega_{\bX^-/\bX^0}\right)\right]_0[-a].\]
Its $n$th cohomology is the weight-zero pushforward of the
$\mathrm{Tor}_{a-n}$ term tensored with
$\omega_{\bX^-/\bX^0}$.

Each such Tor module is finite over
\[\mathscr O(Y^-) =\mathscr O(T^*\bX^0)[y_1,\ldots,y_b,\xi_1,\ldots,\xi_a].\]
Choose finitely many generators homogeneous for the Euler action.
Since each added variable has weight $-1$, only monomials of one
fixed total degree can carry a given generator to a prescribed
weight. There are finitely many such monomials. Each weight space is
therefore finite over $\mathscr O(T^*\bX^0)$, which proves coherence.

Give $\tH^n(K_0)$ the filtration $G$ in which $G_\ell$ is the image
of $F_\ell K_0^n\cap\ker d^n$ in cohomology. With the filtrations on
cycles and boundaries induced from $K_0^n$, we have
\[\mathrm{im}(\gr^F d^{n-1}) \subseteq\gr^F(\mathrm{im}d^{n-1}) \subseteq\gr^F(\ker d^n)\subseteq\ker(\gr^F d^n).\]
The quotient of the middle two terms is $\gr^G\tH^n(K_0)$, so it is
a subquotient of $\tH^n(\gr^F K_0)$
(\textit{cf.}\ \cite[Tag 012K]{StacksFilteredComplexes}). It is therefore
coherent, and $G$ is good. This gives the first support inclusion.

The Tor sheaves are supported on the closed $\gm$-invariant set
$Y^-\cap\siso(\mathcal M)$. For every point of this set, its
limit under $t\to\infty$ is its image under $p^-$ and still belongs
to $\siso(\mathcal M)$. Thus the pushforward is supported on the
closed set $T^*\bX^0\cap\siso(\mathcal M)$, which gives the second
inclusion. Finally, $0\leq k\leq a+b$ in $K^{a-k}$, giving the
asserted cohomological range.
\end{proof}

\bibliographystyle{abbrv}
\newpage
\bibliography{References}

@article{VaragnoloVasserot2011Canonical,
  author        = {Michela Varagnolo and Eric Vasserot},
  title         = {Canonical bases and {K}hovanov--{L}auda algebras},
  journal       = {Journal f{\"u}r die reine und angewandte Mathematik},
  volume        = {659},
  pages         = {67--100},
  year          = {2011},
  doi           = {10.1515/CRELLE.2011.068},
  eprint        = {0901.3992},
  archivePrefix = {arXiv},
  primaryClass  = {math.RT}
}

@article{FujitaHernandez2026,
  author  = {Ryo Fujita and David Hernandez},
  title   = {Monoidal {J}antzen filtrations},
  journal = {Advances in Mathematics},
  volume  = {495},
  pages   = {110963},
  year    = {2026},
  doi     = {10.1016/j.aim.2026.110963}
}

@article{Williamson2016LocalHodge,
  author  = {Geordie Williamson},
  title   = {Local {H}odge theory of {S}oergel bimodules},
  journal = {Acta Mathematica},
  volume  = {217},
  number  = {2},
  pages   = {341--404},
  year    = {2016},
  doi     = {10.1007/s11511-017-0146-8},
  eprint  = {1410.2028},
  archivePrefix = {arXiv},
  primaryClass  = {math.RT}
}

@misc{Ciubotaru2026Jantzen,
  author        = {Dan Ciubotaru},
  title         = {Jantzen filtrations for graded affine Hecke algebras},
  year          = {2026},
  eprint        = {2609.03651},
  archivePrefix = {arXiv},
  primaryClass  = {math.RT},
  doi           = {10.48550/arXiv.2609.03651},
  url           = {https://arxiv.org/abs/2609.03651}
}

@article{Waldspurger2003Plancherel,
  author  = {Waldspurger, Jean-Loup},
  title   = {La formule de {P}lancherel pour les groupes r{\'{e}}ductifs $p$-adiques (d'apr{\`{e}}s {H}arish-{C}handra)},
  journal = {Journal of the Institute of Mathematics of Jussieu},
  year    = {2003},
  volume  = {2},
  number  = {2},
  pages   = {235--333},
  doi     = {10.1017/S147474800300008X}
}

@article{Gabriel1972,
  author  = {Gabriel, Peter},
  title   = {Unzerlegbare {D}arstellungen {I}},
  journal = {Manuscripta mathematica},
  year    = {1972},
  volume  = {6},
  number  = {1},
  pages   = {71--103},
  doi     = {10.1007/BF01298413}
}

@article{lapid2018geometric,
  title={Geometric conditions for $\square$-irreducibility of certain representations of the general linear group over a non-archimedean local field},
  author={Lapid, Erez and M{\'\i}nguez, Alberto},
  journal={Advances in Mathematics},
  volume={339},
  pages={113--190},
  year={2018},
  publisher={Elsevier}
}

@article{KazhdanLusztig1987,
  author  = {Kazhdan, David and Lusztig, George},
  title   = {Proof of the {D}eligne-{L}anglands conjecture for {H}ecke algebras},
  journal = {Inventiones mathematicae},
  year    = {1987},
  volume  = {87},
  number  = {1},
  pages   = {153--215},
  doi     = {10.1007/BF01389157}
}

@book{ChrissGinzburg1997,
  author    = {Chriss, Neil and Ginzburg, Victor},
  title     = {Representation Theory and Complex Geometry},
  publisher = {Birkh{\"a}user},
  address   = {Boston},
  year      = {1997},
  isbn      = {978-0-8176-3792-7}
}

@article {KKKO,
    AUTHOR = {Kang, Seok-Jin and Kashiwara, Masaki and Kim, Myungho and Oh,
              Se-jin},
     TITLE = {Monoidal categorification of cluster algebras},
   JOURNAL = {J. Amer. Math. Soc.},
  FJOURNAL = {Journal of the American Mathematical Society},
    VOLUME = {31},
      YEAR = {2018},
    NUMBER = {2},
     PAGES = {349--426},
      ISSN = {0894-0347,1088-6834},
   MRCLASS = {13F60 (16Gxx 17B37 18D10 81R50)},
  MRNUMBER = {3758148},
MRREVIEWER = {Fan\ Qin},
       DOI = {10.1090/jams/895},
       URL = {https://doi.org/10.1090/jams/895},
}

@article{HernandezLeclerc2010,
  author  = {Hernandez, David and Leclerc, Bernard},
  title   = {Cluster algebras and quantum affine algebras},
  journal = {Duke Mathematical Journal},
  year    = {2010},
  volume  = {154},
  number  = {2},
  pages   = {265--341},
  doi     = {10.1215/00127094-2010-040}
}

@article{GeissLeclercSchroer2006,
  author  = {Gei{\ss}, Christof and Leclerc, Bernard and Schr{\"o}er, Jan},
  title   = {Rigid modules over preprojective algebras},
  journal = {Inventiones mathematicae},
  year    = {2006},
  volume  = {165},
  number  = {3},
  pages   = {589--632},
  doi     = {10.1007/s00222-006-0507-y}
}

@article{Aubert1995Dualite,
  author  = {Aubert, Anne-Marie},
  title   = {Dualit{\'e} dans le groupe de {G}rothendieck de la cat{\'e}gorie des repr{\'e}sentations lisses de longueur finie d'un groupe r{\'e}ductif $p$-adique},
  journal = {Transactions of the American Mathematical Society},
  volume  = {347},
  number  = {6},
  pages   = {2179--2189},
  year    = {1995},
  doi     = {10.1090/S0002-9947-1995-1285969-0}
}

@article{KashiwaraSaito1997,
  author  = {Kashiwara, Masaki and Saito, Yoshihisa},
  title   = {Geometric construction of crystal bases},
  journal = {Duke Mathematical Journal},
  year    = {1997},
  volume  = {89},
  number  = {1},
  pages   = {9--36},
  doi     = {10.1215/S0012-7094-97-08902-X}
}

@book{KashiwaraSchapira1990Sheaves,
  author    = {Kashiwara, Masaki and Schapira, Pierre},
  title     = {Sheaves on Manifolds},
  series    = {Grundlehren der mathematischen Wissenschaften},
  volume    = {292},
  note      = {With a chapter in French by Christian Houzel},
  publisher = {Springer-Verlag},
  address   = {Berlin},
  year      = {1990},
  doi       = {10.1007/978-3-662-02661-8}
}

@article{EvensMirkovic1997Fourier,
  author  = {Evens, Sam and Mirkovi{\'c}, Ivan},
  title   = {Fourier transform and the {I}wahori-{M}atsumoto involution},
  journal = {Duke Mathematical Journal},
  volume  = {86},
  number  = {3},
  pages   = {435--464},
  year    = {1997},
  doi     = {10.1215/S0012-7094-97-08613-0}
}

@article{MoeglinWaldspurger1986Zelevinski,
  author  = {M{\oe}glin, Colette and Waldspurger, Jean-Loup},
  title   = {Sur l'involution de {Z}elevinski},
  journal = {Journal f{\"u}r die reine und angewandte Mathematik},
  volume  = {372},
  pages   = {136--177},
  year    = {1986}
}

@article{deligne1971theorie2,
  author    = {Deligne, Pierre},
  title     = {Th{\'{e}}orie de {H}odge: {II}},
  journal   = {Publications Math{\'{e}}matiques de l'IH{\'{E}}S},
  volume    = {40},
  pages     = {5--57},
  year      = {1971},
  publisher = {Institut des Hautes {\'{E}}tudes Scientifiques}
}

@article{griffiths1968periods1,
  author    = {Griffiths, Phillip A.},
  title     = {On the Periods of Integrals on Algebraic Manifolds, {I} ({{\it {P}}eriod {D}}istributions and {L}ocal {T}orelli {T}heorem)},
  journal   = {American Journal of Mathematics},
  volume    = {90},
  number    = {2},
  pages     = {568--626},
  year      = {1968},
  publisher = {Johns Hopkins University Press}
}

@article{Schnell2014OverviewMHM,
  author        = {Schnell, Christian},
  title         = {An overview of {M}orihiko {S}aito's theory of mixed {H}odge modules},
  journal       = {arXiv preprint arXiv:1405.3096},
  year          = {2014},
  eprint        = {1405.3096},
  archiveprefix = {arXiv},
  primaryclass  = {math.AG}
}

@unpublished{Sabbah2019IntroMHM,
  author       = {Sabbah, Claude},
  title        = {Introduction to mixed {H}odge modules},
  note         = {Lecture notes, Angers summer school},
  year         = {2019},
  url          = {https://perso.pages.math.cnrs.fr/users/claude.sabbah/livres/sabbah_angers1904.pdf}
}

@article{crawleyboevey2000exceptional,
  title={On the exceptional fibres of {K}leinian singularities},
  author={Crawley-Boevey, William},
  journal={American Journal of Mathematics},
  volume={122},
  number={5},
  pages={1027--1037},
  year={2000},
  publisher={Johns Hopkins University Press}
}

@article{laumon1985transformations,
  author    = {Laumon, G{\'e}rard},
  title     = {Transformations canoniques et sp{\'e}cialisation pour les $\mathcal{D}$-modules filtr{\'e}s},
  journal   = {Ast{\'e}risque},
  volume    = {130},
  pages     = {56--129},
  year      = {1985},
  publisher = {Soci{\'e}t{\'e} math{\'e}matique de France}
}

@article{Saito1988,
  author  = {Saito, Morihiko},
  title   = {Modules de {H}odge Polarisables},
  journal = {Publications of the Research Institute for Mathematical Sciences},
  year    = {1988},
  volume  = {24},
  number  = {6},
  pages   = {849--995},
  doi     = {10.2977/prims/1195173930}
}

@article{Saito1990,
  author  = {Saito, Morihiko},
  title   = {Mixed {H}odge {M}odules},
  journal = {Publications of the Research Institute for Mathematical Sciences},
  year    = {1990},
  volume  = {26},
  number  = {2},
  pages   = {221--333},
  doi     = {10.2977/prims/1195171082}
}

@article{Lusztig1991Quivers,
  author  = {Lusztig, George},
  title   = {Quivers, perverse sheaves, and quantized enveloping algebras},
  journal = {Journal of the American Mathematical Society},
  volume  = {4},
  number  = {2},
  pages   = {365--421},
  year    = {1991},
  doi     = {10.1090/S0894-0347-1991-1088333-2}
}

@inproceedings{Ginzburg1987,
  author    = {Ginzburg, Victor},
  title     = {Geometrical aspects of representation theory},
  booktitle = {Proceedings of the International Congress of Mathematicians},
  volume    = {1},
  pages     = {840--848},
  year      = {1987},
  publisher = {American Mathematical Society},
  address   = {Providence, RI}
}

@article{Suzuki1998,
  author    = {Suzuki, Takeshi},
  title     = {Rogawski's conjecture on the {J}antzen filtration for the degenerate affine {H}ecke algebra of type {A}},
  journal   = {Representation Theory of the American Mathematical Society},
  volume    = {2},
  pages     = {393--409},
  year      = {1998},
  doi       = {10.1090/S1088-4165-98-00043-0}
}

@misc{Dro25,
      title={Poles of intertwining operators in terms of irreducible components of {L}usztig's characteristic variety in type {A}}, 
      author={Johannes Droschl},
      year={2025},
      eprint={2508.13817},
      archivePrefix={arXiv},
      primaryClass={math.RT},
      url={https://arxiv.org/abs/2508.13817}, 
}

@article {Dat05,
    AUTHOR = {Dat, J.-F.},
     TITLE = {{$v$}-tempered representations of {$p$}-adic groups. {I}.
              {$l$}-adic case},
   JOURNAL = {Duke Math. J.},
  FJOURNAL = {Duke Mathematical Journal},
    VOLUME = {126},
      YEAR = {2005},
    NUMBER = {3},
     PAGES = {397--469},
      ISSN = {0012-7094,1547-7398},
   MRCLASS = {22E50 (11F70)},
  MRNUMBER = {2120114},
MRREVIEWER = {Anton\ Deitmar},
       DOI = {10.1215/S0012-7094-04-12631-4},
       URL = {https://doi.org/10.1215/S0012-7094-04-12631-4},
}

@article {LapMin25,
    AUTHOR = {Lapid, Erez and M\'inguez, Alberto},
     TITLE = {A binary operation on irreducible components of {L}usztig's
              nilpotent varieties {II}: applications and conjectures for
              representations of {${\rm GL}_n$} over a non-archimedean local
              field},
   JOURNAL = {Pure Appl. Math. Q.},
  FJOURNAL = {Pure and Applied Mathematics Quarterly},
    VOLUME = {21},
      YEAR = {2025},
    NUMBER = {2},
     PAGES = {813--863},
      ISSN = {1558-8599,1558-8602},
   MRCLASS = {22E50 (14M15 16G20 16T20 20G42)},
  MRNUMBER = {4847251},
       DOI = {10.4310/pamq.241205005734},
       URL = {https://doi.org/10.4310/pamq.241205005734},
}

@article{Zel80,
author = {Zelevinsky, A. V.},
journal = {Annales scientifiques de l'École Normale Supérieure},
number = {2},
pages = {165-210},
publisher = {Elsevier},
title = {Induced representations of reductive $p$-adic groups. {II}. {O}n irreducible representations of $\mathrm{GL}(n)$},
volume = {13},
year = {1980}
}

@article {BerZel77,
    AUTHOR = {Bernstein, I. N. and Zelevinsky, A. V.},
     TITLE = {Induced representations of reductive {${\mathfrak p}$}-adic
              groups. {I}},
   JOURNAL = {Ann. Sci. \'{E}cole Norm. Sup. (4)},
  FJOURNAL = {Annales Scientifiques de l'\'{E}cole Normale Sup\'{e}rieure.
              Quatri\`eme S\'{e}rie},
    VOLUME = {10},
      YEAR = {1977},
    NUMBER = {4},
     PAGES = {441--472},
      ISSN = {0012-9593},
   MRCLASS = {22E50},
  MRNUMBER = {579172},
       URL = {http://www.numdam.org/item?id=ASENS_1977_4_10_4_441_0},
}

@article {BerZel76,
    AUTHOR = {Bernstein, I. N. and Zelevinsky, A. V.},
     TITLE = {Representations of the group {$GL(n,F),$} where {$F$} is a
              local non-{A}rchimedean field},
   JOURNAL = {Uspehi Mat. Nauk},
  FJOURNAL = {Akademija Nauk SSSR i Moskovskoe Matemati\v{c}eskoe
              Ob\v{s}\v{c}estvo. Uspehi Matemati\v{c}eskih Nauk},
    VOLUME = {31},
      YEAR = {1976},
    NUMBER = {3(189)},
     PAGES = {5--70},
      ISSN = {0042-1316},
   MRCLASS = {22E50},
  MRNUMBER = {425030},
MRREVIEWER = {G.\ I.\ Ol\cprime shanski\u{\i}},
}

@article{BushnellKutzko1998,
  author  = {Bushnell, Colin J. and Kutzko, Philip C.},
  title   = {Smooth representations of reductive $p$-adic groups: structure theory via types},
  journal = {Proceedings of the London Mathematical Society},
  volume  = {77},
  number  = {3},
  pages   = {582--634},
  year    = {1998},
  doi     = {10.1112/s0024611598000574}
}

@article{BushnellKutzko1999,
  author  = {Bushnell, Colin J. and Kutzko, Philip C.},
  title   = {Semisimple types in $\mathrm{GL}_n$},
  journal = {Compositio Mathematica},
  volume  = {119},
  number  = {1},
  pages   = {53--97},
  year    = {1999},
  doi     = {10.1023/A:1001773929735}
}

@article{braden2003hyperbolic,
  author    = {Braden, Tom},
  title     = {Hyperbolic localization of intersection cohomology},
  journal   = {Transformation Groups},
  volume    = {8},
  number    = {3},
  pages     = {209--216},
  year      = {2003},
  publisher = {Springer},
  doi       = {10.1007/s00031-003-0719-7},
  eprint    = {math/0202251},
  archivePrefix = {arXiv},
  primaryClass  = {math.AG}
}

@article{drinfeld2014theorem,
  author        = {Drinfeld, Vladimir and Gaitsgory, Dennis},
  title         = {On a theorem of {B}raden},
  journal       = {arXiv preprint arXiv:1308.3786},
  year          = {2013},
  eprint        = {1308.3786},
  archivePrefix = {arXiv},
  primaryClass  = {math.AG}
}

@article{Rogawski1985,
  author  = {Rogawski, Jonathan D.},
  title   = {On modules over the {H}ecke algebra of a $p$-adic group},
  journal = {Inventiones mathematicae},
  year    = {1985},
  volume  = {79},
  number  = {3},
  pages   = {443--465},
  doi     = {10.1007/BF01388580}
}

@article{FinisLapidMuller2012,
  author  = {Finis, Tobias and Lapid, Erez and M{\"u}ller, Werner},
  title   = {On the degrees of matrix coefficients of intertwining operators},
  journal = {Pacific Journal of Mathematics},
  year    = {2012},
  volume  = {260},
  number  = {2},
  pages   = {433--456},
  doi     = {10.2140/pjm.2012.260.433}
}

@book{Jantzen1979,
  author    = {Jantzen, Jens Carsten},
  title     = {Moduln mit einem h{\"o}chsten Gewicht},
  series    = {Lecture Notes in Mathematics},
  volume    = {750},
  publisher = {Springer-Verlag},
  address   = {Berlin, Heidelberg, New York},
  year      = {1979},
  doi       = {10.1007/BFb0069521},
  isbn      = {978-3-540-09558-3}
}

@book{Achar2021Perverse,
  author    = {Achar, Pramod N.},
  title     = {Perverse Sheaves and Applications to Representation Theory},
  series    = {Mathematical Surveys and Monographs},
  volume    = {258},
  publisher = {American Mathematical Society},
  address   = {Providence, RI},
  year      = {2021},
  isbn      = {978-1-4704-5597-2},
  doi       = {10.1090/surv/258}
}

@article{BBD1982,
  author    = {Beilinson, Alexander A. and Bernstein, Joseph and Deligne, Pierre},
  title     = {Faisceaux pervers},
  journal   = {Ast{\'e}risque},
  volume    = {100},
  publisher = {Soci{\'e}t{\'e} Math{\'e}matique de France},
  address   = {Paris},
  year      = {1982}
}

@article{AizLap,
  author       = {Aizenbud, Avraham and Lapid, Erez},
  title        = {A binary operation on irreducible components of {L}usztig's nilpotent varieties {I}: definition and properties},
  journal      = {Pure and Applied Mathematics Quarterly},
  year         = {2025},
  volume       = {21},
  number       = {1},
  pages        = {5--41},
  doi          = {10.4310/pamq.241203024730}
}

@article{Zelevinsky1981,
  author  = {Zelevinsky, A. V.},
  title   = {A {$p$}-adic analogue of the {K}azhdan-{L}usztig conjecture},
  journal = {Functional Analysis and Its Applications},
  year    = {1981},
  volume  = {15},
  number  = {2},
  pages   = {83--92},
  doi     = {10.1007/BF01082260}
}

@article{Webster2020Hecke,
  author = {Webster, Ben},
  title = {On graded presentations of {H}ecke algebras and their generalizations},
  journal = {Algebraic Combinatorics},
  volume = {3},
  number = {1},
  year = {2020},
  pages = {1--38},
  doi = {10.5802/alco.84}
}

@article{LaudaVazirani2011,
  author = {Lauda, Aaron D. and Vazirani, Monica},
  title = {Crystals from categorified quantum groups},
  journal = {Advances in Mathematics},
  volume = {228},
  number = {2},
  year = {2011},
  pages = {803--861},
  doi = {10.1016/j.aim.2011.06.009}
}

@article{Laumon1987Fourier,
  author = {Laumon, G{\'e}rard},
  title = {Transformation de {F}ourier, constantes d'{\'e}quations fonctionnelles et conjecture de {W}eil},
  journal = {Publications Math{\'e}matiques de l'IH{\'E}S},
  volume = {65},
  year = {1987},
  pages = {131--210},
  doi = {10.1007/BF02698937}
}

@article{Sun2017Comparison,
  author = {Sun, Shenghao},
  title = {Generic base change, {A}rtin's comparison theorem, and the decomposition theorem for complex {A}rtin stacks},
  journal = {Journal of Algebraic Geometry},
  volume = {26},
  year = {2017},
  pages = {513--555},
  doi = {10.1090/jag/683}
}

@article{ChenDirks2023,
  author = {Chen, Qianyu and Dirks, Bradley},
  title = {On {$V$}-filtration, {H}odge filtration and {F}ourier transform},
  journal = {Selecta Math. (N.S.)},
  volume = {29},
  number = {4},
  year = {2023},
  pages = {Paper No. 50},
  doi = {10.1007/s00029-023-00850-1}
}

@misc{StacksFilteredComplexes,
  key = {Stacks Project},
  title = {{The Stacks Project}},
  howpublished = {\url{https://stacks.math.columbia.edu/tag/012K}}
}

@misc{KhanRaviConcentration,
  author = {Khan, Adeel A. and Ravi, Charanya},
  title = {Categorical concentration},
  year = {2026},
  note = {arXiv:2308.01652v2},
  url = {https://arxiv.org/abs/2308.01652v2}
}

@book{LamModulesRings,
  author = {Lam, T. Y.},
  title = {Lectures on Modules and Rings},
  series = {Graduate Texts in Mathematics},
  volume = {189},
  publisher = {Springer},
  year = {1999}
}

@article{Gurevich2023Simple,
  author = {Gurevich, Maxim},
  title = {Simple modules for quiver {H}ecke algebras and the {R}obinson--{S}chensted--{K}nuth correspondence},
  journal = {Journal of the London Mathematical Society},
  volume = {107},
  number = {2},
  year = {2023},
  pages = {704--749},
  doi = {10.1112/jlms.12695}
}

@article{KKKO2015Simplicity,
  author = {Kang, Seok-Jin and Kashiwara, Masaki and Kim, Myungho and Oh, Se-jin},
  title = {Simplicity of heads and socles of tensor products},
  journal = {Compositio Mathematica},
  volume = {151},
  number = {2},
  year = {2015},
  pages = {377--396},
  doi = {10.1112/S0010437X14007799}
}

@article{Sun2012Decomposition,
  author = {Sun, Shenghao},
  title = {Decomposition theorem for perverse sheaves on {A}rtin stacks over finite fields},
  journal = {Duke Mathematical Journal},
  volume = {161},
  number = {12},
  year = {2012},
  pages = {2297--2310},
  doi = {10.1215/00127094-1723657}
}
\end{document}